\documentclass[12pt]{article}
\catcode `!=11

\newdimen\squaresize 
\newdimen\thickness 
\newdimen\Thickness
\newdimen\ll! \newdimen \uu! 
\newdimen\dd! \newdimen \rr! \newdimen \temp!

\def\sq!#1#2#3#4#5{%
\ll!=#1 \uu!=#2 \dd!=#3 \rr!=#4
\setbox0=\hbox{%
 \temp!=\squaresize\advance\temp! by .5\uu!
 \rlap{\kern -.5\ll! 
 \vbox{\hrule height \temp! width#1 depth .5\dd!}}%
 \temp!=\squaresize\advance\temp! by -.5\uu!  
 \rlap{\raise\temp! 
 \vbox{\hrule height #2 width \squaresize}}%
 \rlap{\raise -.5\dd!
 \vbox{\hrule height #3 width \squaresize}}%
 \temp!=\squaresize\advance\temp! by .5\uu!
 \rlap{\kern \squaresize \kern-.5\rr! 
 \vbox{\hrule height \temp! width#4 depth .5\dd!}}%
 \rlap{\kern .5\squaresize\raise .5\squaresize
 \vbox to 0pt{\vss\hbox to 0pt{\hss $#5$\hss}\vss}}%
}
 \ht0=0pt \dp0=0pt \box0
}

\def\vsq!#1#2#3#4#5\endvsq!{\vbox to \squaresize
  {\hrule width\squaresize height 0pt%
\vss\sq!{#1}{#2}{#3}{#4}{#5}}}

\newdimen \LL! \newdimen \UU! \newdimen \DD! \newdimen \RR!

\def\vvsq!{\futurelet\next\vvvsq!}
\def\vvvsq!{\relax
  \ifx     \next l\LL!=\Thickness \let\continue!=\skipnexttoken!
  \else\ifx\next u\UU!=\Thickness \let\continue!=\skipnexttoken!
  \else\ifx\next d\DD!=\Thickness \let\continue!=\skipnexttoken!
  \else\ifx\next r\RR!=\Thickness \let\continue!=\skipnexttoken!
  \else\ifx\next P\let\continue!=\place!
  \else\def\continue!{\vsq!\LL!\UU!\DD!\RR!}%
  \fi\fi\fi\fi\fi 
  \continue!}

\def\skipnexttoken!#1{\vvsq!}

\def\place! P#1#2#3{%
\rlap{\kern.5\squaresize\temp!=.5\squaresize\kern#1\temp!
  \temp!=\squaresize \advance\temp! by #2\squaresize
 \temp!=.5\temp!  \raise\temp!\vbox to 0pt%
{\vss\hbox to 0pt{\hss$#3$\hss}\vss}}\vvsq!}

\def\Young#1{\LL!=\thickness \UU!=\thickness \DD! = \thickness \RR! = \thickness
\vbox{\smallskip\offinterlineskip
\halign{&\vvsq! ## \endvsq!\cr #1}}}

\def\blank{\omit\hskip\squaresize}
\catcode `!=12

\usepackage{tikz,hyperref}

\usepackage{amsthm,amsmath,amssymb}

\usepackage{graphicx}

\newcommand{\arxiv}[1]{\href{http://arxiv.org/abs/#1}{\texttt{arXiv:#1}}}

\usepackage{url}
\usepackage{xcolor}
\usepackage[all,cmtip]{xy}

\theoremstyle{plain}
\newtheorem{theorem}{Theorem}
\newtheorem{lemma}[theorem]{Lemma}
\newtheorem{corollary}[theorem]{Corollary}
\newtheorem{proposition}[theorem]{Proposition}

\theoremstyle{definition}
\newtheorem{definition}[theorem]{Definition}
\newtheorem{example}[theorem]{Example}

\theoremstyle{remark}
\newtheorem{remark}[theorem]{Remark}

\title{\bf Homomesy in products of two chains}

\author{James Propp\thanks{Partially supported by
NSF Grant \#1001905.}\\
\small Department of Mathematics\\[-0.8ex]
\small University of Massachusetts Lowell\\[-0.8ex] 
\small Lowell, MA, U.S.A.\\
\small \url{http://jamespropp.org}\\
\and
Tom Roby\\
\small Department of Mathematics\\[-0.8ex]
\small University of Connecticut\\[-0.8ex]
\small Storrs, CT, USA\\
\small \url{http://www.math.uconn.edu/~troby}
}

\newcommand{\cred}[1]{{\color{red}#1}}
\newcommand{\cblu}[1]{{\color{blue}#1}}

\newcommand{\cA}{{\mathcal{A}}}
\newcommand{\cO}{{\mathcal{O}}}
\newcommand{\cP}{{\mathcal{P}}}
\newcommand{\cS}{{\mathcal{S}}}
\newcommand{\cW}{{\mathcal{W}}}
\newcommand{\bR}{{\mathbf{R}}}
\newcommand{\bC}{{\mathbf{C}}}
\newcommand{\bZ}{{\mathbf{Z}}}
\newcommand{\bN}{{\mathbf{N}}}
\newcommand{\rowmotion}{\Phi}
\newcommand{\promotion}{\partial}
\newcommand{\symmdiff}{\bigtriangleup}
\newcommand{\outdeg}{{\rm outdeg}}
\newcommand{\fstat}{f}
\newcommand{\ra}{\stackrel{\tau}{\rightarrow}}
\newcommand{\da}{{\scriptstyle f} \downarrow \ \ \, }

\def\urltilde{\kern -.15em\lower .7ex\hbox{\~{}}\kern .04em}

\def\inv{\mathop{\rm inv}}

\def\SSYT{\mathop{\rm SSYT}}
\def\id{\mathop{\rm id}}
\def\eset{\emptyset}
\def\bc#1#2{\left(\kern -2pt{#1\atop #2} \kern -2pt\right)}

\def\ds{\displaystyle}

\begin{document}

\maketitle

\begin{abstract}
Many invertible actions $\tau$ on a set $\mathcal{S}$ of combinatorial objects,
along with a natural statistic $f$ on $\mathcal{S}$, exhibit the following
property which we dub \textbf{homomesy}:
the average of $f$ over each $\tau$-orbit in $\mathcal{S}$ is the
same as the average of $f$ over the whole set $\mathcal{S}$.  
This phenomenon was first noticed by Panyushev in 2007 in the context of
the rowmotion action on the set of antichains of a root poset;
Armstrong, Stump, and Thomas proved Panyushev's conjecture in 2011. 
We describe a theoretical framework for results of this kind 
that applies more broadly, giving examples in a variety of contexts.  
These include linear actions on vector spaces, sandpile dynamics, 
Suter's action on certain subposets of Young's Lattice, Lyness 5-cycles,
promotion of rectangular semi-standard Young tableaux,
and the rowmotion and promotion actions on certain posets.  
We give a detailed description of the latter situation 
for products of two chains.  

\end{abstract}

\section{Introduction}
\label{sec:int}
We begin with the definition of our main unifying concept,
and supporting nomenclature.
\begin{definition}\label{def:ce}
Given a set $\cS$, 
an invertible map $\tau$ from $\cS$ to itself such that each
$\tau$-orbit is finite, 
and a function (or ``statistic'') $\fstat: \cS \rightarrow K$
taking values in some field $K$ of characteristic zero,
we say 
the triple $(\cS,\tau,\fstat)$ 
exhibits {\bf homomesy}\footnote{Greek for ``same middle''} 
if there exists a constant $c \in K$
such that for every $\tau$-orbit $\cO \subset \cS$
\begin{equation}
\label{general-ce}
\frac{1} {\#\cO}
\sum_{x \in \cO} \fstat(x) = c .
\end{equation}
In this situation
we say that the function $f: \cS \rightarrow K$ is 
{\bf homomesic}
under the action of $\tau$ on $\cS$,
or more specifically \textbf{\textit{c}-mesic}.
\end{definition}
When $\cS$ is a finite set,
homomesy can be restated equivalently as 
all orbit-averages being equal to the global average:

\begin{equation}
\label{ce}
\frac{1}{\#\cO} \sum_{x \in \cO} \fstat(x) =
\frac{1}{\#\cS} \sum_{x \in \cS} \fstat(x).
\end{equation}
We will also apply the term homomesy more broadly to include the case
that the statistic $\fstat$ takes values
in a vector space over a field of characteristic 0
(as in sections \ref{subsec-linear} and \ref{ssec:sand}).

We have found many instances of (\ref{ce})
where $\cS$ is a finite collection of combinatorial objects 
(e.g., order ideals in a poset),
$\tau$ is a natural action on $\cS$ (e.g., rowmotion or promotion),
and $\fstat$ is a natural measure on $\cS$ (e.g., cardinality).
Many (but far from all) situations that support examples of homomesy also
support examples of the cyclic sieving phenomenon 
of Reiner, Stanton, and White~\cite{RSW04},
and more exploration of the links and differences is certainly in order.  
At the stated level of generality the notion of homomesy appears to be new, 
but specific instances can be found in earlier literature.  
In particular, Panyushev~\cite{Pan09} conjectured and Armstrong, 
Stump, and Thomas~\cite{AST11} proved the following homomesy result: 
if $\cS$ is the set of antichains in the
root poset of a finite Weyl group, $\rowmotion$ is the operation variously
called the Brouwer-Schrijver map~\cite{BS74}, the Fon-der-Flaass
map~\cite{Fon93,CF95}, the reverse map~\cite{Pan09}, Panyushev
complementation~\cite{AST11}, and rowmotion~\cite{SW12}, and $\fstat(A)$
is the cardinality of the antichain $A$, then $(\cS,\rowmotion,\fstat)$
satisfies (\ref{ce}).

Our main results for this paper involve studying the 
rowmotion action and also the (Striker-Williams) promotion action 
associated with the poset $P=[a]\times [b]$. 
(See Section~\ref{sec:chains} for precise definitions.
Note that we use $[n]$ to denote both the set $\{1,\dots,n\}$
and the natural poset with those elements, according to context.)
We show that the statistic
$\fstat :=\#A$, the size of the antichain, is homomesic with respect to
the promotion action, and that the statistic $\fstat =\#I(A)$, the
size of the corresponding order ideal, is homomesic with respect to both
the promotion and rowmotion actions.  

Although these results are of intrinsic interest,
we think the main contribution of the paper is its
identification of homomesy as a phenomenon that occurs quite widely. 
Within any linear space of functions on $\cS$,
the functions that are 0-mesic under $\tau$,
like the functions that are invariant under $\tau$,
form a subspace. There is a loose sense 
in which the notions of invariance and homomesy
(or, more strictly speaking, 0-mesy) are complementary;
an extremely clean case of this complementarity
is outlined in subsection~\ref{subsec-linear},
and a related complementarity 
(in the context of continuous rather than discrete orbits)
is sketched in subsection~\ref{subsec-circle}.
This article gives a general overview of the broader picture 
as well as a few specific examples done in more detail for the operators 
of promotion and rowmotion associated with the poset $[a]\times [b]$.  

We provide examples of homomesy in a wide variety of contexts.  
These include the following actions with corresponding statistics,
each of which is explained in more detail in the indicated subsections. 
All the examples in Section~\ref{sec:eg} are fairly independent of each other
and of our main new results in Section~\ref{sec:chains}, 
so the reader may focus on some examples more than others, according to taste.  

\begin{enumerate}
\item reversal of permutations with the statistic 
that counts inversions~[\S~\ref{subsec-inversions}];

\item cyclic rotation of words on $\{-1, +1 \}$ with 
the $\{0,1\}$-function that indicates whether a word satisfies 
the ballot condition~[\S~\ref{subsec-ballot}]; 

\item cyclic rotation of words on $\{-1, +1 \}$ with the statistic that counts 
the number of (multiset) inversions in the word~[\S~\ref{subsec-bits}]; 

\item linear maps which satisfy $T^{n}=1$ acting in a vector space $V$ 
with statistic the identity function~[\S~\ref{subsec-linear}]; 

\item the phase-shift action on simple harmonic motion
with statistics given by certain polynomial combinations
of position and velocity~[\S~\ref{subsec-circle}];  

\item the Lyness 5-cycle acting on (most of) $\bR^{2}$ 
with $f((x,y)) = \log |x^{-1} + x^{-2}|$ 
as the statistic~[\S~\ref{ssec:five}]; 

\item the action on recurrent sandpile configurations given by adding 1 grain
to the source vertex and then allowing the system to stabilize, 
with statistic the firing vector~[\S~\ref{ssec:sand}]; 

\item Suter's action on Young diagrams with a weighted 
cardinality statistic~[\S~\ref{ssec:suter}];

\item promotion in the sense of Sch\"utzenberger acting on 
semistandard Young tableaux of rectangular shape with statistic 
given by summing the entries in any centrally-symmetric
subset of cells of the tableaux~[\S~\ref{ssec:SSYT}],
as studied by Bloom, Pechenik, and Saracino~\cite{BPS13};

\item promotion (in the sense of \cite{SW12}) acting on the set 
of order ideals of $[a]\times [b]$ with 
the cardinality statistic~[\S~\ref{subsec-promotion}];

\item rowmotion acting on the set of order ideals of $[a]\times [b]$ 
with the cardinality statistic~[\S~\ref{sss:rmJP}]; and

\item rowmotion acting on the set of antichains of $[a]\times [b]$ 
with the cardinality statistic~[\S~\ref{sss:rmAP}].  

\end{enumerate}

The authors are grateful to Omer Angel, Drew Armstrong, Anders Bj\"orner, 
Robin Chapman, Joyce Chu, Barry Cipra, Karen Edwards, Robert Edwards, 
Darij Grinberg, Shahrzad Haddadan, Andrew Hone, Mike Joseph, Greg Kuperberg, 
Svante Linusson, Vic Reiner, Ralf Schiffler, Richard Stanley, Jessica Striker, 
Nathan Williams, Peter Winkler and Ben Young for useful conversations.  
Mike LaCroix wrote fantastic postscript code to generate animations 
and pictures that illustrate our maps operating on order ideals
on products of chains (Figures~\ref{fig:pm32a},~\ref{fig:pm32b},
and~\ref{fig:rm42}).  Ben Young also provided a diagram which we 
modified for Figure~\ref{fig:aw75}.  
Darij Grinberg's eagle eye caught many errors and
opportunities for improved exposition.  
An anonymous referee made very helpful suggestions for improving the
``extended abstract'' version of this paper that was presented at the
25th annual conference on Formal Power Series and Algebraic
Combinatorics, held in Paris in June 2013.  Another anonymous referee provided
stimulating ideas as well as very helpful recommendations for improving
the exposition in the journal version of the article.  Several of our ideas 
were first incubated at meetings of the long-running Cambridge Combinatorics
and Coffee Club (CCCC), organized by Richard Stanley.  

\section{Examples of Homomesy}\label{sec:eg}

Here we give a variety of examples of homomesy in combinatorics,
the first two of which long predate the general notion of homomesy;
we also give non-combinatorial examples
that establish links with other branches of mathematics.
For examples of homomesy associated with 
piecewise-linear maps and birational maps, see~\cite{EP13}.

\subsection{Inversions in permutations}
\label{subsec-inversions}

Let $\cS$ be the set of permutations of $\{1,2,\dots,n \}$,
let $\tau$ send $\pi_1\pi_{2}\dots\pi_n$ 
(a permutation written in one-line notation)
to its reversal $\pi_n\pi_{n-1}\dots\pi_1$
and let $\fstat(\pi)$ be the number of inversions in $\pi$.
Since $\tau^2$ is the identity,
and since $\fstat(\pi)+\fstat(\tau(\pi))=n(n-1)/2$,
$\fstat$ is $c\,$-mesic under the action of $\tau$,
where $c=n(n-1)/4$.

\subsection{Ballot theorems}
\label{subsec-ballot}

Fix two nonnegative integers $a$ and $b$ and set $n=a+b$.  
Let $\cS$ be the set of words $(s_1,s_2,\dots,s_n)$ of length $n$,
consisting of $a$ letters equal to $-1$ and $b$ letters equal to $+1$; 
we think of each such word as an order 
for counting $n$ ballots in a two-way election,
$a$ of which are for candidate A and $b$ of which are for candidate B.
If $a < b$, then candidate B will be deemed the winner
once all $a+b$ ballots have been counted,
and we ask for the probability that 
at every stage in the counting of the ballots candidate B is in the lead.
This probability is the same as the expected value of $\fstat(s)$,
where $\fstat(s)$ is 1 
if $s_1+\dots+s_i>0$ for all $1 \leq i \leq n$ and is 0 otherwise, 
and where $s$ is chosen uniformly at random from $\cS$.
Bertrand's Theorem states that this probability is $(b-a)/(b+a)$.

Dvoretzky and Motzkin's famous ``cycle lemma'' proof 
of Bertrand's Theorem~\cite{DM47} 
(see also Raney's lemma described on page 346 of~\cite{GKP})
may be recast in our framework as follows:

\begin{proposition}
\label{prop:ballot}
Let $\tau: = C_{L}: \cS \rightarrow \cS$ 
be the leftward cyclic shift operator 
that sends $(s_1,s_2,s_3,\dots,s_n)$ to $(s_2,s_3,\dots,s_n,s_1)$.
Then over any orbit $\mathcal{O}$ one has
\[
\frac{1}{\#\cO}\sum_{s \in \cal{O}} \fstat (s) = \frac{b-a}{b+a}\,. 
\]
In other words, $f$ is $c$-mesic with $c=\frac{b-a}{b+a}$.
\end{proposition}
\noindent
See~\cite{R07} for details.


\subsection{Inversions in two-element multiset permutations}

\label{subsec-bits}

As in the preceding section, let $\cS$ be the set of words of length $n=a+b$
consisting of $a$ letters equal to $-1$ and $b$ letters equal to $+1$
(without the requirement that $a < b$),
and let $\fstat (s):= \inv (s):=\#\{i<j: s_{i}>s_{j}\}$.
For fixed $i<j$, the number of $s$ in $\cS$ with $s_{i}>s_{j}$
(i.e., with $s_i=1$ and $s_j=-1$) is ${{n-2} \choose {a-1}}$,
so the probability that an $s$ chosen uniformly at random from $\cS$
satisfies $s_i > s_j$ is 
${{n-2} \choose {a-1}}/{n \choose a} = \frac{ab}{n(n-1)}$;
hence by additivity of expectation, the expected value of $\inv(s)$ is
$c = \sum_{i<j} \frac{ab}{n(n-1)} = \frac{n(n-1)}{2} \frac{ab}{n(n-1)} = ab/2$.
Indeed, $\fstat$ is $c$-mesic under the action 
of the involution on $\cS$ that reverses the order of the letters
(which gives us another way to compute the expected value of $\inv(s)$).
Here we give a less trivial example of homomesy.

\begin{proposition}
\label{prop:bits}
Let $\tau$ be the left-shift $C_L$ on $\cS$ and $\fstat (s):= \inv (s)$
as above. Then over each orbit
$\mathcal{O}$ we have 
\[
\frac{1}{\#\cal{O}} \sum_{s \in \cal{O}} \fstat (s) = 
\frac{ab}{2}=\frac{1}{\#\cal{S}} \sum_{s \in \cal{S}} \fstat (s).
\]
In other words, the inversion statistic is $c$-mesic under the action of
cyclic rotation, with $c=ab/2$.
\end{proposition}

One way to prove Proposition~\ref{prop:bits} is
to rewrite the indicator function of $(s_{i},s_{j})$ being an inversion pair as
$\frac{1}{4}(1+s_{i})(1-s_{j})$.  Then
\begin{eqnarray*}
\fstat(s) & = & \sum_{i<j} (1+s_i)(1-s_j)/4
= \frac14 \sum_{i<j} (1 + s_i - s_j - s_i s_j) \\
& = & \frac14 \left( 
\sum_{i<j} 1 + \sum_{i<j} s_i - \sum_{i<j} s_j - \sum_{i<j} s_i s_j \right)\,.
\end{eqnarray*}
In the final expression, the first and fourth sums are independent of $s$, 
since for all $s$, $\sum_{i<j} 1$ is $\frac{n(n-1)}{2}$
and $\sum_{i<j} s_i s_j$ is 
$$\left( \frac{a(a-1)}{2}+\frac{b(b-1)}{2} \right)(+1)
+\left( ab \right) (-1) = \frac{n(n-1)}{2} - 2ab,$$
so $$\fstat(s) = \frac14 \left( 2ab + \sum_{i<j} s_i - \sum_{i<j} s_j \right).$$
Since the average value of $\sum_{i<j} s_i$ over each cyclic orbit 
equals the average value of $\sum_{i<j} s_j$ over that orbit\footnote{One
way to see it is to count how often a given $s_k$ occurs when we sum
the sums $\sum_{i<j} s_i$ over a given cyclic orbit. It is easy to
see that $s_k$ occurs $0$ times in one such sum, $1$ times in another,
$2$ times in another, etc., for a total of $0 + 1 + \ldots + (n-1)$
times; but the same can be said of the sum $\sum_{i<j} s_j$.},
these terms cancel, 
so that the average value of $\fstat$ over each orbit is $ab/2$.

In the particular case $a=b=2$, the six-element set $\cS$ decomposes into two
orbits, shown in Figure~\ref{fig:pm42}.  (Here we recode the elements
of $\cS$ as ordinary bit-strings, representing $+1$ and $-1$ by 1 and 0,
respectively.) As frequently happens, not all
orbits are the same size.  But one may also view the orbit of size 2
as being part of a ``superorbit'' of size 4, cycling through the same set of
elements twice.

A different proof (in keeping with 
the ``equivariant bijection philosophy'' discussed 
in subsection~\ref{subsec-equivariant})
associates with each $s\in \cS$
the set $\tilde{s} \subseteq \{1,2,\dots,n\}$
consisting of the positions $1 \leq i \leq n$ for which $s_i = -1$.
The collection $\tilde{\cS}$ of such sets $\tilde{s}$ is precisely
the set of $a$-element subsets of $\{1,2,\dots,n\}$.
Then the action of $\tau$ on $\cS$ is isomorphic to
the action of $\tilde{\tau}$ on $\tilde{\cS}$,
where applying $\tilde{\tau}$ to $\tilde{s}$
decrements each element by 1 mod $n$.
Likewise, the inversion statistic $\fstat$ on $\cS$
corresponds to the statistic $\tilde{\fstat}$ on $\tilde{\cS}$, where
$$\tilde{\fstat}(\tilde{s}) 
=  \left( \sum_{i \in \tilde{s}} i \right) - \left( 1+2+\cdots+a \right)
= \left( \sum_{i \in \tilde{s}} i \right) - \frac{a(a+1)}{2}.$$
Although the orbits of this action can have different sizes, 
each must be of size $d$ where $d\mid n$.  
So we can repeat such an orbit $n/d$ times to form a superorbit of
length $n$, which has the same average for any statistic as the original orbit.
Now each of the $a$ members of the set $\tilde{s}$
takes on each value in $\{1,2,\dots,n\}$ 
over the $\tilde{\tau}$-superorbit of $\tilde{s}$,
so that
$$\sum_{i=0}^{n-1} \tilde{f}(\tilde{\tau}^i \tilde{s})
= a(1+2+\cdots+n) - n \, \frac{a(a+1)}{2}
= \frac{an(n+1)}{2} - \frac{an(a+1)}{2}$$
and
$$\frac{1}{n} \sum_{i=0}^{n-1} \tilde{f}(\tilde{\tau}^i \tilde{s})
= \frac{a(n+1)}{2} - \frac{a(a+1)}{2} = \frac{a(n-a)}{2}.$$
It follows that $(\tilde{\cS},\tilde{\tau},\tilde{\fstat})$, 
along with $(\cS,\tau,\fstat)$,
is $c$-mesic with $c = a(n-a)/2 = ab/2$.

A third way to prove Proposition~\ref{prop:bits} is to derive it from our Theorem
\ref{promotion-ideals}; see Remark~\ref{rem:corprop-bits}.

\begin{figure}[t]
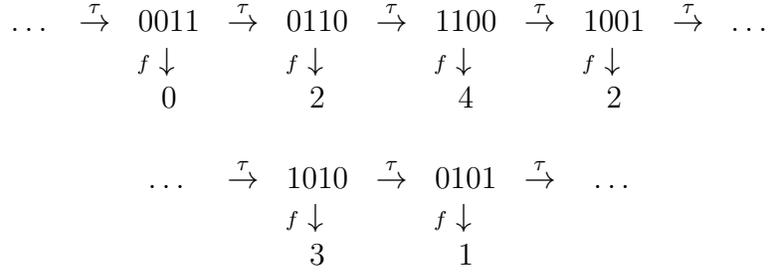

$$
\begin{array}{ccccccccccc}
\dots & \ra &  0011 & \ra & 0110 & \ra & 1100 & \ra &  1001 & \ra & \dots \\
      &     &  \da  &     & \da  &     & \da  &     &  \da  &     &  \\
      &     &   0   &     &   2  &     &   4  &     &    2  &     &  \\
      &     &       &     &      &     &      &     &       &     &  \\
      &     & \dots & \ra & 1010 & \ra & 0101 & \ra & \dots &     &  \\
      &     &       &     & \da  &     & \da  &     &       &     &  \\
      &     &       &     &   3  &     &   1  &     &       &     &
\end{array}
$$
\caption{The two orbits of the action of the cyclic shift
on binary strings consisting of two 0's and two 1's.
The average value of the inversion statistic 
is $(0+2+4+2)/4 = 2$ on the orbit of size 4
and $(3+1)/2 = 2$ on the orbit of size 2.}
\label{fig:pm42}
\end{figure}

\subsection{Linear actions on vector spaces}
\label{subsec-linear}

Let $V$ be a (not necessarily finite-dimensional) vector space
over a field $K$ of characteristic zero,
and define $\fstat(v)=v$
(that is, our ``statistic'' is just the identity function).
Let $T: V \rightarrow V$
be a linear map such that $T^n = I$ (the identity map on $V$)
for some fixed $n \geq 1$.
Say $v$ is \textit{invariant} under $T$ if $Tv=v$, 
and \textit{0-mesic} under $T$ if $(v+Tv+\cdots+T^{n-1}v)/n=0$.

\begin{proposition}
Every $v \in V$ can be written uniquely
as the sum of an invariant vector $\overline{v}$
and a 0-mesic vector $\hat{v}$. 
\end{proposition}

\begin{proof}
One can check that $v = \overline{v} + \hat{v}$ is such a decomposition,
with $\overline{v} = (v+Tv+\cdots+T^{n-1}v)/n$ and $\hat{v} = v - \overline{v}$,
and no other such decomposition is possible because that would yield a 
nonzero vector that is both invariant and 0-mesic, which does not exist. 
\end{proof}

In representation-theoretic terms, we are applying symmetrization to $v$
to extract from it the invariant component $\overline{v}$ associated with 
the trivial representation of the cyclic group, and the homomesic (0-mesic)
component $\hat{v}$ consists of everything else.

One suggestive way of paraphrasing the above is:
Every element of the kernel of $I-T^n = (I-T)(I+T+T^2+\dots+T^{n-1})$
can be written uniquely as the sum of an element of the kernel of $I-T$
and an element of the kernel of $I+T+T^2+\dots+T^{n-1}$.

This picture relates more directly to our earlier definition
if we use the dual space $V^*$ of linear functionals on $V$
as the set of statistics on $V$.  As a concrete example,
let $V = \bR^n$ and let $T$ be the cyclic shift of coordinates
sending $(x_1,x_2,\ldots,x_n)$ to $(x_n,x_1,\ldots,x_{n-1})$.  
The $T$-invariant functionals form a 1-dimensional subspace of $V^*$
spanned by the functional $(x_1,x_2,\ldots,x_n) \mapsto x_1+x_2+\cdots+x_n$,
while the 0-mesic functionals form an $(n-1)$-dimensional subspace of $V^*$
spanned by the $n-1$ functionals
$(x_1,x_2,\ldots,x_n) \mapsto x_i-x_{i+1}$ (for $1 \leq i \leq n-1$).
Also, we can consider the ring $\bR[x_1,\ldots,x_n]$ of polynomial functions 
$p(x_1,x_2,\dots,x_n)$ on $\bR^n$;
this ring, viewed as a vector space over $\bR$,
can be written as the direct sum of 
the subspace of polynomials that are invariant under the action of $T$
and the subspace of polynomials that are 0-mesic under the action of $T$.

\subsection{A circle action}
\label{subsec-circle}

Let $\cS$ be the set of (real-valued) functions $f(t)$ 
satisfying the differential equation $f''(t) + f(t) = 0$, 
that is, the set of functions of the form $f(t) = A \sin (t - \phi)$,
where $A$ is the amplitude and $\phi$ is the initial phase.
Then we have $f'' = -f$, $f''' = -f'$, $f'''' = f$, etc.,
so that every polynomial function of
$f, f', f'', f''', f'''', \dots$,
can be written as a polynomial in $f$ and $f'$.
Evolving $f$ in time is tantamount to shifting the phase $\phi$.

Given an element $p(x,y)$ of the ring $\bR[x,y]$,
we will say $p$ is \textbf{invariant under time-evolution} 
(or, more compactly, that $p$ is an \textbf{invariant})
if $\frac{d}{dt} \, p(f(t),f'(t)) = 0$ for all $f$ in $\cS$,
and $\mathbf{c}\,$\textbf{-mesic} 
if $\frac{1}{2\pi} \int_0^{2\pi} p(f(t),f'(t))\,dt = c$  for all $f$ in $\cS$.
For example, $x^2+y^2$ is invariant and $x$ and $y$ are 0-mesic;
one can think of the first quantity as the total energy of a harmonic oscillator
and the second and third as the mean displacement and mean velocity.

We can give a basis for $\bR[x,y]$, viewed as a vector space $V$ over $\bR$,
consisting of the nonnegative powers of $x^2+y^2$
(which jointly span the subspace of $V$
consisting of all polynomials that are invariant under time-evolution), 
along with the functions $x$, $y$, $xy$, $x^2-y^2$, 
$x^3$, $x^2 y$, $x y^2$, $y^3$, etc.\ 
(which jointly span the 0-mesic subspace of $V$).

\begin{proposition}
\label{prop:circle}
Let $V_n$ be the $(n+1)$-dimensional vector subspace of $\bR[x,y]$
spanned by the monomials $x^a y^b$ with $a+b=n$.
When $n$ is odd, all of $V_n$ is 0-mesic.
When $n$ is even, $V_n$ can be written as the direct sum
of an $n$-dimensional subspace of 0-mesic functions
and a 1-dimensional subspace of functions
that are invariant under time-evolution.
\end{proposition}

\begin{proof}
Define $\int_{S^{1}} p(x,y) = 
\frac{1}{2\pi} \int_0^{2\pi} p(\cos t,\ \sin t)\:dt$, 
where $S^{1}$ is the unit circle in $\bR^{2}$.  
Consider the monomial $x^a y^b$ with $a+b=n$.
If $a$ (resp.\ $b$) is odd, 
the involution $(x,y) \mapsto (-x,y)$ (resp.\ $(x,-y)$) shows that 
$\int_{S^{1}} x^a y^b = 0$ (using merely that $\sin$ is odd and $\cos$ is even).
If $a$ and $b$ are both even,
then $\int_{S^{1}} x^a y^b$ is some positive number $c_{a,b}$.
Now let $a,b$ vary subject to $a+b=n$.
If $n$ is odd, then $x^a y^b$ is 0-mesic
for all $a,b$ with $a+b=n$ (since at least one of $a,b$ is odd),
so all of $V_n$ is 0-mesic.
If $n$ is even, then for $a,b$ even,
$(1/c_{a,b}) x^a y^b - (1/c_{n,0}) x^{n} y^{0}$ is 0-mesic,
and these functions span an $(n/2)$-dimensional space;
adding in the 0-mesic functions $x^a y^b$ with $a,b$ odd (and $a+b=n$),
we get an $n$-dimensional space of 0-mesies.

Finally, we must verify that 
the $n$-dimensional space of 0-mesies linearly complements 
the 1-dimensional space of invariants spanned by $(x^2+y^2)^{n/2}$.
First we note that (as in subsection~\ref{subsec-linear})
every function that is both invariant (under time-evolution)
and homomesic must be constant;
for, any polynomial function $p(\cdot,\cdot)$
such that the value of $p(A \cos t, B \sin t)$
is independent of $t$ (invariance)
and independent of $A$ and $B$ (homomesy)
must be constant.
It follows that the only function in $V_n$
that is both invariant and 0-mesic is the constant function 0.
Hence the subspace of $V_n$ spanned by 0-mesies
and the subspace of $V_n$ spanned by invariants are linearly disjoint.
Complementarity then follows from a dimension-count.
\end{proof}

Here, as in the preceding section, we get
a clean complementarity between invariance and homomesy.
That is, every element in $\bR[x,y]$ can be written uniquely 
as the sum of an invariant element and a 0-mesic element.
One way to see this abstractly is to introduce
an action of the circle-group on $\bR[x,y]$
that is compatible with the action of the circle-group on $\cS$
and our interpretation of $\bR[x,y]$ as a space of functions.
Specifically, define $(T_\theta) p(x,y) = 
p(x \cos \theta - y \sin \theta, x \sin \theta + y \cos \theta)$,
and let $L$ be the linear map that sends $f \in \bR[x,y]$ to $g$
where $g(x,y)$ is the average of $T_\theta f(x,y)$ 
as $\theta$ varies of $[0,2\pi]$.
Then the subspace $V_h$ of homomesies is the kernel of $L$
and the subspace $V_i$ of invariants is the image of $L$.

Also, we can look at the linear maps $\phi$ from $\bR[x,y]$ to $\bR$ 
that are are unaffected by this action,
in the sense that $\phi \circ T_t = \phi$ for all $t$.
In this context we might call $V_i$ the ``coinvariant kernel'', 
since $V_i$ is the mutual kernel
of all invariant linear maps $\phi$ from $\bR[x,y]$ to $\bR$.

\subsection{5-cycles}\label{ssec:five}

Let $U$ be the set of all $(x,y)$ in $\bR^2$ with
$x$, $y$, $x+1$, $y+1$, and $x+y+1$ all nonzero.
The map $\tau:U \rightarrow U$ sending $(x,y)$ to $(y,(y+1)/x)$ has order 5. 
We can recursively define a sequence 
$(x_{1},x_{2}, \dotsc  )$ by $x_{1}:=x$, $x_{2}:=y$ and 
the (Lyness) recurrence $x_{i-1} x_{i+1} = x_i + 1$, 
so that $\tau (x_{i-1},x_{i}) = (x_{i}, x_{i+1})$.  
This sequence turns out to have period 5, thereby giving rise to the
\textbf{Lyness 5-cycle}
\[
x \leadsto y \leadsto (y+1)/x \leadsto (x+y+1)/xy 
\leadsto (x+1)/y \leadsto x \,.
\]
This is associated with the $A_2$ cluster algebra,
e.g., by way of four-rowed frieze patterns.
(One accessible article on frieze patterns is~\cite{Pro08},
although it lacks references to many relevant articles
written in the past decade.)
Let $\fstat((x,y)) = \log |h(x)|$
where $h(z) = z^{-1} + z^{-2}$ is well-defined and nonzero throughout $U$.

\begin{proposition}
\label{prop:lyness}
The function $\fstat$ is 0-mesic under the action of $\tau$ on $U$.
\end{proposition}

\begin{proof} (Andy Hone)
Using the fact that $x_{i-1} x_{i+1} = x_i + 1$ 
with all subscripts interpreted mod 5 (this is just the Lyness recurrence),
we write the product $h(x_1) h(x_2) h(x_3) h(x_4) h(x_5)$ as
$\prod ( x_i + 1 ) / x_i^2 = \prod x_{i+1} x_{i-1} / x_i^2$,
and the numerator and denominator factors all cancel,
showing that the product is 1.
\end{proof}

Applying the map $z \mapsto z/(1+z)$ to the Lyness 5-cycle
we obtain the ``Bloch 5-cycle'' 
\[
x\leadsto y\leadsto (1-x)/(1-xy)\leadsto 1-xy\leadsto (1-y)/(1-xy) \leadsto x
\]
satisfying the recurrence $x_{i-1} + x_{i+1} = x_{i-1} x_i x_{i+1} + 1$.
For example, the Lyness 5-cycle\\[2pt]  
$\ds \left(1,3,4,\frac{5}{3},\frac{2}{3}\right)$ maps to the Bloch 5-cycle
$\ds \left(\frac{1}{2}, \frac{3}{4}, \frac{4}{5}, 
\frac{5}{8}, \frac{2}{5}\right)$.\\[2pt]  

If we let $U'$ be the set of all $(x,y)$ in $\bR^2$ 
with $x,y \not\in \{0,1\}$ and $xy \neq 1$,
then the map that sends $(x,y)$ to $(y,(1-x)/(1-xy))$
is an order-5 map from $U'$ to itself,
and the Bloch-Wigner function on $\bC \setminus \{0,1\}$
(a variant of the dilogarithm function; see~\cite{W14})
is 0-mesic under this action.

\subsection{Sandpile dynamics}\label{ssec:sand}

Let $G$ be a finite directed graph with vertex set $V$.
For $v \in V$ let $\outdeg(v)$ be 
the number of directed edges emanating from $v$,
and for $v,w \in V$ let $\deg(v,w)$ be
the number of directed edges from $v$ to $w$
(which we will permit to be larger than 1, even when $v=w$).
Define the {\em combinatorial Laplacian} of $G$ as the matrix $\Delta$
(with rows and columns indexed by the vertices of $V$)
whose $v,v$th entry is $\outdeg(v)-\deg(v,v)$
and whose $v,w$th entry for $v \neq w$ is $-\deg(v,w)$.
Specify a vertex $t$ with the property that 
for all $v \in V$ there is a forward path from $v$ to $t$, 
called the {\em global sink};
let $V^- = V \setminus \{t\}$,
and let $\Delta'$ (the reduced Laplacian) be the matrix $\Delta$
with the row and column associated with $t$ removed.
By the Matrix-Tree theorem, $\Delta'$ is nonsingular.
A {\em sandpile configuration} on $G$ (with sink at $t$)
is a function $\sigma$ from $V^-$ to the nonnegative integers.
(For more background on sandpiles, see Holroyd, Levine, M\'esz\'aros,
Peres, Propp, and Wilson~\cite{HLMPPW08}.)
We say $\sigma$ is {\em stable} 
if $\sigma(v) < \outdeg(v)$ for all $v \in V^-$.
For any sandpile configuration $\sigma$,
Dhar's least-action principle for sandpile dynamics 
(see Levine and Propp~\cite{LP10}) tells us
that the set of nonnegative-integer-valued functions $u$ on $V^-$ 
such that $\sigma-\Delta' u$ is stable
has a unique minimal element $\phi=\phi(\sigma)$ 
in the natural (pointwise) ordering;
we call $\phi$ the {\em firing vector} for $\sigma$
and we call $\sigma-\Delta' \phi$ the {\em stabilization} of $\sigma$,
denoted by $\sigma^\circ$.
If we choose a {\em source vertex} $s \in V^-$,
then we can define an action on sandpile configurations
via $\tau(\sigma) = (\sigma+1_s)^\circ$,
where $1_v$ denotes the function that takes the value 1 at $v$ and 0 elsewhere.
Say that $\sigma$ is {\em recurrent} (relative to $s$) 
if $\tau^m(\sigma) = \sigma$ for some $m > 0$.  
(This notion of recurrence is slightly weaker than that of
\cite{HLMPPW08}; they are equivalent when every vertex is reachable
by a path from $s$.)
Then $\tau$ restricts to an invertible map
from the set of recurrent sandpile configurations to itself.
Let $f(\sigma) = \phi(\sigma+1_s)$.
Since $\tau(\sigma) = \sigma + 1_s - \Delta' f(\sigma)$
we have $\tau(\sigma) - \sigma = 1_s - \Delta' f(\sigma)$;
if we average this relation over all $\sigma$ in a particular $\tau$-orbit,
the left side telescopes, giving $0 = 1_s - \Delta' \overline{f}$,
where $\overline{f}$ denotes the average of $f$ over the orbit.
Hence:

\begin{proposition}
\label{prop:fvhm}
Under the action of $\tau$ on recurrent sandpile configurations
described above,
the function $f: \sigma \mapsto \phi(\sigma+1_s)$ is homomesic,
and its orbit-average is the function $f^*$ on $V^-$
such that $\Delta' f^* = 1_s$ (unique because $\Delta'$ is nonsingular).
\end{proposition}

\begin{example}\label{eg:sand}

Figure~\ref{fig:sand}
shows an example of the $\tau$-orbits for the case where $G$ is 
the bidirected cycle graph with vertices 1, 2, 3, and 4,
with a directed edge from $i$ to $j$ iff $i-j = \pm 1$ mod 4; here 
the discrete Laplacian is 
\[
\Delta =
\left(
\begin{array}{rrrr}
 2 & -1 &  0 & -1 \\
-1 &  2 & -1 &  0 \\
 0 & -1 &  2 & -1 \\
-1 &  0 & -1 &  2 
\end{array} \right)\,.  
\]
Let the source be $s=2$ and global sink be $t=4$.
The sandpile configuration $\sigma$
is represented by the triple $(\sigma(1),\sigma(2),\sigma(3))$.
The four recurrent configurations $\sigma$ are
$(1,0,1)$, $(1,1,1)$, $(0,1,1)$, and $(1,1,0)$,
and the respective firing vectors $f(\sigma)$ are
$(0,0,0)$, $(1,2,1)$, $(0,1,1)$, and $(1,1,0)$.
The average value of the firing vector statistic $f$
is $f^*=(\frac12,1,\frac12)$ on each orbit.
Treating $f^*$ as a column vector and multiplying on the left by $\Delta'$
gives the column vector $(0,1,0)=1_{s}$:
\[
\left(
\begin{array}{rrr}
 2 & -1 &  0  \\
-1 &  2 & -1  \\
 0 & -1 &  2 
\end{array} \right)
\left(
\begin{array}{c}
1/2 \\
 1  \\
1/2 \\
\end{array} \right)
=
\left(
\begin{array}{c}
0 \\
1 \\
0 \\
\end{array} \right)\,.
\]
\end{example}

\begin{figure}[t]
$$
\begin{array}{ccccccccccc}
\dots & \ra & (1,0,1) & \ra & (1,1,1) & \ra & \dots \\
      &     &   \da   &     &   \da   &     &   \\
      &     & (0,0,0) &     & (1,2,1) &     &   \\
      &     &         &     &         &     &   \\
\dots & \ra & (0,1,1) & \ra & (1,1,0) & \ra & \dots \\
      &     &   \da   &     &   \da   &     &   \\
      &     & (0,1,1) &     & (1,1,0) &     &   \\
      &     &         &     &         &     &   \\
\end{array}
$$
\caption{The two orbits in the action of the sandpile map $\tau$
on recurrent configurations on the cycle graph of size 4,
with source at 2 and sink at 4.  There are two orbits,
each of size 2, and the average of $f$ along each orbit
is $(1/2,1,1/2)$.}
\label{fig:sand}
\end{figure}

We should mention that in this situation
all orbits are of the same cardinality.
This is a consequence of the fact that
the set of recurrent sandpile configurations
can be given the structure of a finite abelian group
(the ``sandpile group'' of $G$).
For, given any finite abelian group $G$ and any element $h \in G$,
the action of $h$ on $G$ by multiplication
has orbits that are precisely the cosets of $G/H$,
where $H$ is the subgroup of $G$ generated by $h$,
and all these cosets have size $|H|$.

Similar instances of homomesy were known for 
a variant of sandpile dynamics called rotor-router dynamics;
see Holroyd-Propp~\cite{HP10}.  
It was such instances of homomesy
that led the second author to seek instances of the phenomenon
in other, better-studied areas of combinatorics.

\subsection{Suter's action on Young diagrams} \label{ssec:suter}

In~\cite{Su02}, Suter described
an action of the dihedral group $D_n$ ($n \geq 1$)
on a particular subgraph $Y_n$ of the Hasse diagram of Young's lattice.  
This specializes to an action of the cyclic group $C_n$.
Let the {\em hull} of a Young diagram
be the smallest rectangular diagram that contains it, 
and let $Y_{n}$ be the set of all Young diagrams 
whose hulls are contained in the staircase diagram $(n-1,n-2,\dots,1)$.  
Here we will consider only the cyclic action generated by 
the invertible operation $\rho_n$ defined by Suter as follows:
Given a Young diagram $\lambda \in Y_{n}$ 
(drawn ``French'' style as rows of boxes in the first quadrant)
we discard the boxes in the bottom row (let us say there are $k$ of them),
move all the remaining boxes one step downward and to the right,
and insert a column of $n-1-k$ boxes at the left.
Suter shows that the resulting diagram $\mu$ is again in $Y_n$
and that the map $\rho_n: \lambda \mapsto \mu$ is invertible.
For example, the action of $\rho_{5}$ on
$Y_{5}$ produces the following four orbits:  
\[
\squaresize = 9pt   \thickness = 1pt 
\left( \eset\,, \quad \Young{\cr \cr \cr \cr}\,,\quad  
\Young{&  \cr & \cr  & \cr}\,, \quad 
\Young{& & \cr & & \cr}\,, \quad \Young{ & & & \cr} \right) \,, \qquad 
\left( \Young{\cr}\,, \quad \Young{\cr \cr \cr }\,,\quad  
\Young{\cr & \cr  & \cr}\,, \quad 
\Young{& \cr & & \cr}\,, \quad \Young{  & & \cr} \right) \,, 
\]
\[
\squaresize = 9pt \squaresize = 9pt   \thickness = 1pt 
\left( \Young{\cr \cr }\,, \quad \Young{\cr \cr & \cr }\,,\quad  
\Young{& \cr & \cr}\,, \quad 
\Young{\cr & & \cr}\,, \quad \Young{  & \cr} \right) \,, \qquad \left( 
\Young{\cr & \cr} \right)\,. 
\]

Figure~\ref{fig:suter} shows another example with $n=6$ and $k=2$,
where boldface black numbers correspond to boxes that get shifted
when one passes from $\lambda=(2,2,1,1)$ to $\rho_{6}(\lambda)=(3,2,2)$.
Suter shows that the map $\rho_n$ 
is an automorphism of the undirected graph $Y_n$,
and that $\rho_{n}^n$ is the identity on $Y_n$.

\begin{figure}[t]
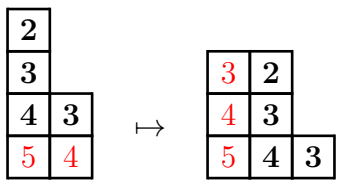


\[
\squaresize = 16pt \thickness = 1pt
\Young{\mathbf{2} &\blank  & \blank \cr
\bf{3} & \blank   & \blank         \cr 
\bf{4} & \bf{3} & \blank $\mapsto$ \cr
\cred    5  &  \cred    4  &  \blank \cr }
\quad 
\Young{
\cred 3 & \bf{2} &\blank \cr	 
\cred 4 & \bf{3} &\blank \cr	   
\cred 5 & \bf{4} & \bf{3} \cr }
\]
	
\caption{An example of Suter's map for $n=6$ and $k=2$}
\label{fig:suter}
\end{figure}

Let $f$ be the statistic on $Y_n$
that sends each Young diagram to the sum 
of the weights of its constituent boxes,
where the box at the lower left has weight $n-1$,
its two neighbors have weight $n-2$, and so on. 
The boxes in Figure~\ref{fig:suter} have been marked with their weights,
so we can see that $f(\lambda)=5+4+4+3+3+2=21$
while $f(\rho_{6}(\lambda))=5+4+4+3+3+3+2=24$.

\begin{proposition}
\label{prop:suter}
Under the action of $\rho_{n}$ on $Y_n$,
the function $f$ is $c$-mesic
with $c = (n^3-n)/12$.
\end{proposition}
\noindent For example, 
the weights corresponding to each orbit of $\rho_5$ on $Y_5$ (shown above) are
\[
(0,10,15,15,10), \quad (4,9,14,14,9), \quad (7,12,12,12,7), \quad (10)\,, 
\]
each of which has average $10 = (5^3-5)/12$.  
\begin{proof} (David Einstein)
The proposition follows from a more refined assertion,
in which we take positive integers $i,j$ with $i+j=n$
and only look at the sum of weights of the boxes of weight $i$
and the boxes of weight $j$; call this $f_{i,j}$.
We claim that $f_{i,j}$ is homomesic with average $ij$.
Then since $f = (f_{1,n-1} + f_{2,n-2} + \cdots + f_{n-1,1})/2$,
it will follow that $f$ is homomesic with average
$\frac12 ((1)(n-1) + (2)(n-2) + \cdots + (n-1)(1)) = (n-1)(n)(n+1)/12 =
(n^3-n)/12$.  

Note that the diagonal slides in the definition of $\rho_{n}$
do not affect the weight of a cell,
because the weights of the cells are constant
along diagonals of slope $-1$ (see Figure~\ref{fig:suter}).
It takes $j$ diagonal sliding operations to move a cell of weight $i$ 
that starts in the first column so that it disappears,
and likewise with the roles of $i$ and $j$ reversed.
So each cell of weight $i$ or weight $j$ added in the first column 
contributes $ij/n$ to the average of $f_{i,j}$.

The definition of $\rho_n$ shows that in going from $\lambda$ to
$\rho_n(\lambda)$, we gain cells of weights $n-1,n-2,\ldots,k+1$ and lose
cells of weights $n-1,n-2,\ldots,n-k$ (where $k$ is the length of the
first row of $\lambda$). So we lose a cell of weight $j$ if and only
if we don't gain a cell of weight $i$. Thus, when we perform $\rho_n$
a total of $n$ times, the number of cells of weight $j$ lost is $n$
minus the number of cells of weight $i$ gained. But the number of
cells of weight $j$ gained is the number of cells of weight $j$ lost
(what comes in is what comes out).
This means that if $r$ cells of weight $j$ are added in a complete cycle, 
then $n-r$ cells of weight $i$ are added, 
for a total of $n$ cells of weight either $i$ or $j$.  
Thus we get an average of $n(ij/n)=ij$ 
for the sum of the weights of these cells across an orbit.
\end{proof}

It should be noted that for this and similar examples,
our notion of homomesy of cyclic actions
can be adapted in a straightforward fashion
to the action of other finite groups.
In the case of Suter symmetry,
the fact that $f$ is $c$-mesic under the action of the cyclic group
implies that $f$ is $c$-mesic under the action of the dihedral group
(since every orbit of the dihedral group is the union
of two \textit{same-size} orbits of the cyclic group).

\subsection{Rectangular Young tableaux}\label{ssec:SSYT}

For a fixed Young diagram $\lambda$,
let $\SSYT_{k}(\lambda )$ denote the set of 
\textbf{semistandard Young tableaux} of shape $\lambda$ and ceiling $k$, 
i.e., fillings of the cells of $\lambda$ with elements of $[k]$ 
which are \emph{weakly} increasing in each row 
and \emph{strictly} increasing in each column. 
(See, e.g., \cite[\S~7.10]{Sta99} for more information 
about these objects and their relationship to symmetric functions.)  
In the particular case where $\lambda = (n^{m}) : = (n,n,\dotsc ,n)$ 
is a \emph{rectangular} shape with $m$ parts, all equal to $n$, 
the Sch\"utzenberger \emph{promotion} operator $\cP$ satisfies
$\cP^{k}=\id$~\cite[Cor.~5.6]{R10}.   
(Simpler proofs are available for \emph{standard Young tableaux}; 
see e.g.\ \cite[Thm.~4.1(a)]{Sta09} and the references therein.)

Now fix any subset $R$ of the cells of $(n^m)$ 
and for $T\in \SSYT_{k}(n^m) $ set $\sigma_{R}(T)$ to be 
the sum of the entries of $T$ whose cells lie in $R$.  
\begin{theorem}[Bloom-Pechenik-Saracino]\label{thm:ssyt}
Let $k$ be a positive integer and suppose that $R\subseteq (n^m)$ 
is symmetric with respect to 180-degree rotation about the center of $(n^{m})$. 
Then the statistic $\sigma_{R}$ is $c$-mesic with respect to 
the action of promotion on $\SSYT_{k}(n^{m})$, 
with $c=|R|\left(\frac{k+1}{2}\right)$.  
\end{theorem}
For example, consider the following promotion orbit within $\SSYT_{5}(3^{2})$ 
(where our tableaux are now drawn ``English'' style, using matrix coordinates):

\[
\squaresize = 16pt \thickness = 1pt
\Young{\cred 1&1&\cblu 2 & \blank $\mapsto$\cr
	 \cblu 2&3&\cred 4\cr} \quad 
\Young{\cred 1&1&\cblu 3 &\blank $\mapsto$\cr
	 \cblu 2&5&\cred 5\cr} \quad 
\Young{\cred 1&2&\cblu 4 &\blank $\mapsto$\cr
	 \cblu 4&5&\cred 5\cr} \quad 
\Young{\cred 1&3&\cblu 4 &\blank $\mapsto$\cr
	 \cblu 3&4&\cred 5\cr} \quad 
\Young{\cred 2&2&\cblu 3 &\blank $\Lsh$\cr
	 \cblu 3&4&\cred 5\cr} \quad 
\]
Then the sum of the values in the upper left and lower right cells 
(shown in \cred{red}) across the orbit is $(5,6,6,6,7)$, 
which averages to $6 = 2 \left(\frac{5+1}{2}\right)$.
Similarly, the sum of the \cblu{blue} entries 
in the lower left and upper right corners
across the orbit is $(4,5,8,7,6)$, with average 6,
and the sum of the black entries 
in the middles of the two rows
across the orbit is $(4,6,7,7,6)$, with average 6.

This result was stated as a conjecture in several talks given by the authors, 
and recently proved by J.~Bloom, O.~Pechenik, and D.~Saracino~\cite{BPS13}.  
The latter also prove a version of the result for
\emph{cominuscule posets}.  
For the action of \emph{K-promotion} on \emph{increasing tableaux} 
of rectangular shapes, they prove an analogous result for two-rowed shapes, 
and show that it fails in general when $\lambda$ is a rectangle 
with more than two rows.

\section{\!\!Promotion and rowmotion in products of two chains}
\label{sec:chains}

For a finite poset $P$, we let $J(P)$ denote the set of order ideals 
(or down-sets) of $P$, $F(P)$ denote the set of (order) filters 
(or up-sets) of $P$, and $\cA(P)$ be the set of antichains of $P$.  
(For standard definitions and notation about
posets and ideals, see Stanley~\cite{Sta11}.)  There is a bijection
$J(P)\leftrightarrow \cA(P)$ given by taking the maximal elements of $I\in
J(P)$ or conversely by taking the order ideal generated by an antichain
$A\in \cA(P)$.  Similarly, there is a bijection $F(P)\leftrightarrow
\cA(P)$.  Composing these with the complementation bijection 
$I \mapsto \overline{I} = P \setminus I$ from $J(P)$ to $F(P)$ 
leads to an interesting map that has been studied 
in several contexts~\cite{BS74,Fon93,CF95,Pan09,AST11,SW12}, namely 
$\rowmotion_{A}:= \cA(P) \rightarrow J(P) \rightarrow F(P) \rightarrow \cA(P)$
and the companion map
$\rowmotion_{J}:= J(P) \rightarrow F(P) \rightarrow \cA(P) \rightarrow J(P)$,
where the subscript indicates whether we consider the map to be
operating on antichains or order ideals.  We often drop the subscript
and just write $\rowmotion$ when context makes clear which is meant.
Following Striker and Williams~\cite{SW12} we call this map \textbf{rowmotion}.

It should be noted that the maps considered 
by Brouwer, Schrijver, Cameron, Fon-der-Flaass, and Panyushev 
are the inverses 
of the maps considered by Striker, Williams, Armstrong, Stump, Thomas, 
and ourselves; we think that the newer convention is more natural, 
to the extent it is more natural to cycle through the integers mod $n$ 
by repeatedly adding 1 than by repeatedly subtracting 1.

Let $[a]\times [b]$ denote the poset that is a product of chains of
lengths $a$ and $b$. 
Figure~\ref{fig:rm42} shows an orbit of the action of $\rowmotion_{J}$ 
acting on the set of order ideals of the poset $P = [4] \times [2]$
starting from the ideal generated by the antichain $\{(2,1)\}$.  
Note that the elements of $[4]\times [2]$ here are
represented by the \emph{squares} rather than the points in the picture,
with covering relations represented by shared edges.  One can also view
this as an orbit of $\rowmotion_{A}$ if one just considers the
maximal elements in each shaded order ideal.  

This section contains our main specific results, namely that the following
triples exhibit homomesy: 
\[
\left(J([a]\times [b]), \rowmotion_{J}, \#I\right)\,;\qquad 
\left(\cA ([a]\times [b]), \rowmotion_{A}, \#A\right)\,;\text{ and}\qquad 
\left(J([a]\times [b]), \promotion_{J}, \#I\right)\,.
\]
Here $\promotion_{J}$ is the promotion operation 
to be defined in the next subsection, 
and $\#I$ (resp. $\#A$) denotes the statistic on $J(P)$ (resp. $\cA (P)$) 
that is the cardinality of the order ideal $I$ (resp. the antichain $A$).  
All maps operate on the left (e.g., we write $\promotion_J I$,
not $I \promotion_J$).

\subsection{Background on the toggle group}
\label{subsec-toggle}

Several of our examples arise from the \emph{toggle group} 
of a finite poset (first explicitly defined in~\cite{SW12};
see also~\cite{CF95,Sta09,SW12}).  We review some basic facts
and provide some pointers to relevant literature.

\begin{definition}\label{def:tog}

Let $P$ be a poset.  Given $x \in P$, we define the toggle operation 
$\sigma_x: J(P) \rightarrow J(P)$
(``toggling at $x$'') via
$$\sigma_x(I) = \left\{ \begin{array}{ll}
I \symmdiff \{x\} & \mbox{if $I \symmdiff \{x\} \in J(P)$}; \\
I & \mbox{otherwise,} \end{array} \right.$$
where $A \symmdiff B$ denotes the symmetric difference  
$(A \setminus B) \, \cup \, (B \setminus A)$. 
\end{definition}
\begin{proposition}[\cite{CF95}]
\label{prop:togtwo}
Let $P$ be a poset.  
(a) For every $x\in P$, $\sigma_{x}$ is an
involution, i.e., $\sigma_{x}^{2}=1$.\\  
(b) For every $x,y\in P$ where neither $x$ covers $y$ nor $y$ covers
$x$, the toggles commute, i.e., $\sigma_{x}\sigma_{y}=\sigma_{y}\sigma_{x}$.  
\end{proposition}

\begin{proposition}[\cite{CF95}]
\label{prop:togall}
Let $x_1,x_2,\dots,x_n$ be any linear extension (i.e., any
order-preserving listing of the elements) of a poset $P$ with $n$
elements.  Then the composite map $\sigma_{x_1} \sigma_{x_2} \cdots
\sigma_{x_n}$ coincides with the rowmotion operation $\rowmotion_{J}$.  
\end{proposition}
Although we do not use the following corollary, 
it provides context for how we view
rowmotion on a finite graded poset.  
\begin{corollary}[\cite{SW12}, Cor.\ 4.9]\label{cor:togranks}
Let $P$ be a graded poset of rank $r$, and set
$T_{k}:=\prod_{x\text{ has rank } k}\sigma_{x}$, 
the product of all the toggles of elements of fixed rank $k$.  
(This is well-defined by Proposition~\ref{prop:togtwo}.)  
Then the composition $T_{0}T_{1}T_{2}\dotsb
T_{r}$ coincides with $\rowmotion_{J}$, 
i.e., rowmotion is the same as toggling by ranks from top to bottom. 
\end{corollary}

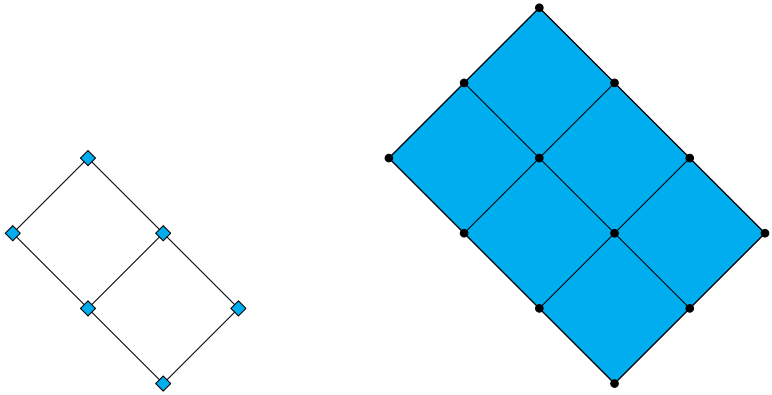
\begin{figure}[t]
\begin{center}
\begin{tikzpicture}
\draw (0,0) -- (-2,2);
\draw (1,1) -- (-1,3);
\draw (0,0) -- (1,1);
\draw (-1,1) -- (0,2);
\draw (-2,2) -- (-1,3);
\draw[fill=cyan] (0.1,0) -- (0,0.1) -- (-0.1,0) -- (0,-0.1) -- (0.1,0);
\draw[fill=cyan] (1.1,1) -- (1,1.1) -- (0.9,1) -- (1,0.9) -- (1.1,1);
\draw[fill=cyan] (-0.9,1) -- (-1,1.1) -- (-1.1,1) -- (-1,0.9) -- (-0.9,1);
\draw[fill=cyan] (0.1,2) -- (0,2.1) -- (-0.1,2) -- (0,1.9) -- (0.1,2);
\draw[fill=cyan] (-1.9,2) -- (-2,2.1) -- (-2.1,2) -- (-2,1.9) -- (-1.9,2);
\draw[fill=cyan] (-0.9,3) -- (-1,3.1) -- (-1.1,3) -- (-1,2.9) -- (-0.9,3);
\draw[fill=cyan] (3,3) -- (6,0) -- (8,2) -- (5,5) -- (3,3);
\draw (3,3) -- (6,0);
\draw (4,4) -- (7,1);
\draw (5,5) -- (8,2);
\draw (3,3) -- (5,5);
\draw (4,2) -- (6,4);
\draw (5,1) -- (7,3);
\draw (6,0) -- (8,2);
\draw[fill=black] (3,3) circle (0.05cm);
\draw[fill=black] (4,2) circle (0.05cm);
\draw[fill=black] (5,1) circle (0.05cm);
\draw[fill=black] (6,0) circle (0.05cm);
\draw[fill=black] (4,4) circle (0.05cm);
\draw[fill=black] (5,3) circle (0.05cm);
\draw[fill=black] (6,2) circle (0.05cm);
\draw[fill=black] (7,1) circle (0.05cm);
\draw[fill=black] (5,5) circle (0.05cm);
\draw[fill=black] (6,4) circle (0.05cm);
\draw[fill=black] (7,3) circle (0.05cm);
\draw[fill=black] (8,2) circle (0.05cm);
\end{tikzpicture}
\end{center}
\caption{The ordinary and modified Hasse diagrams of $[3]\times [2]$}
\label{fig:hasse}
\end{figure}

We focus on the case $P = [a] \times [b]$, 
whose elements we write as $(k,\ell)$.
We depict this poset by sending
$(k,\ell) \in [a] \times [b]$ to $(\ell-k,k+\ell-2) \in \bZ \times \bZ$ 
(for all $1 \leq k \leq a$, $1 \leq \ell \leq b$).
That is, we take the points $(k,\ell) \in \bN \times \bN$,
rotate by 45 degrees counterclockwise while dilating by $\sqrt{2}$,
and then flip points across the vertical axis.
E.g., for the poset $P = [3] \times [2]$
we get the diagram shown in the left panel of Figure~\ref{fig:hasse}. 
This is the usual Hasse diagram for $[a] \times [b]$
(or rather one of the two usual Hasse diagrams,
since one could exchange the roles of $a$ and $b$).
We typically draw a modified diagram in which
the dots in the Hasse diagram are replaced by boxes
(much as the dots in a Ferrers graph of a partition
correspond to boxes in the Young diagram);
see the right panel of Figure~\ref{fig:hasse}.
This modified Hasse diagram makes it easier to see the correspondence
between order ideals and lattice paths
that will be crucial in much of what follows;
see Figure~\ref{fig:path}.
These modified Hasse diagrams contain dots, 
but those dots do not correspond to elements of the poset.

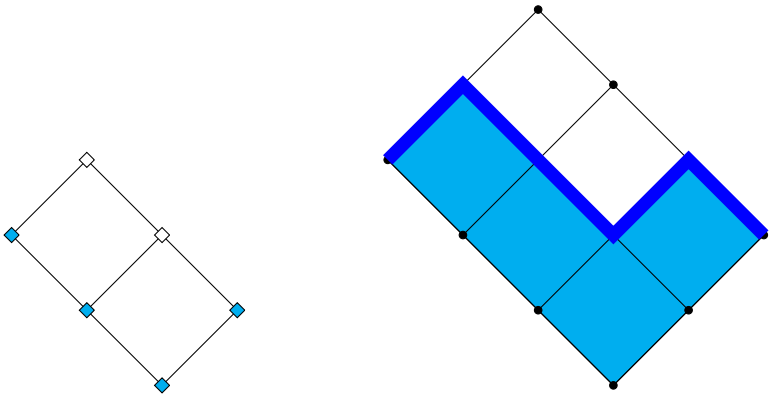
\begin{figure}[t]
\begin{center}
\begin{tikzpicture}
\draw[fill=cyan] (3,3) -- (6,0) -- (8,2) -- (7,3) -- (6,2) -- (4,4) -- (3,3);
\draw (0,0) -- (-2,2);
\draw (1,1) -- (-1,3);
\draw (0,0) -- (1,1);
\draw (-1,1) -- (0,2);
\draw (-2,2) -- (-1,3);
\draw[fill=cyan] (0.1,0) -- (0,0.1) -- (-0.1,0) -- (0,-0.1) -- (0.1,0);
\draw[fill=cyan] (1.1,1) -- (1,1.1) -- (0.9,1) -- (1,0.9) -- (1.1,1);
\draw[fill=cyan] (-0.9,1) -- (-1,1.1) -- (-1.1,1) -- (-1,0.9) -- (-0.9,1);
\draw[fill=white] (0.1,2) -- (0,2.1) -- (-0.1,2) -- (0,1.9) -- (0.1,2);
\draw[fill=cyan] (-1.9,2) -- (-2,2.1) -- (-2.1,2) -- (-2,1.9) -- (-1.9,2);
\draw[fill=white] (-0.9,3) -- (-1,3.1) -- (-1.1,3) -- (-1,2.9) -- (-0.9,3);
\draw (3,3) -- (6,0);
\draw (4,4) -- (7,1);
\draw (5,5) -- (8,2);
\draw (3,3) -- (5,5);
\draw (4,2) -- (6,4);
\draw (5,1) -- (7,3);
\draw (6,0) -- (8,2);
\draw[fill=black] (3,3) circle (0.05cm);
\draw[fill=black] (4,2) circle (0.05cm);
\draw[fill=black] (5,1) circle (0.05cm);
\draw[fill=black] (6,0) circle (0.05cm);
\draw[fill=black] (4,4) circle (0.05cm);
\draw[fill=black] (5,3) circle (0.05cm);
\draw[fill=black] (6,2) circle (0.05cm);
\draw[fill=black] (7,1) circle (0.05cm);
\draw[fill=black] (5,5) circle (0.05cm);
\draw[fill=black] (6,4) circle (0.05cm);
\draw[fill=black] (7,3) circle (0.05cm);
\draw[fill=black] (8,2) circle (0.05cm);
\draw[blue, line width=5pt] (3,3) -- (4,4) -- (5,3) -- (6,2) -- (7,3) -- (8,2);
\end{tikzpicture}
\end{center}
\caption{An order ideal in $[3] \times [2]$
and the associated lattice path}
\label{fig:path}
\end{figure}

\begin{definition}\label{def:rankfile}
For $P = [a] \times [b]$, we call the sets of $(k, \ell)$
with constant $k+\ell-2$ \textbf{ranks}
(in accordance with standard poset terminology),
and the sets of $(k, \ell)$
with constant $\ell-k$ \textbf{files},
sets with constant $k$ \textbf{positive fibers},
and sets with constant $\ell$ \textbf{negative fibers}.  
(One would like to say that ``fiber'' means ``row or column'',
but since Striker and Williams use the words ``row'' and ``column''
to denote what we call ranks and files respectively,
we fear that saying this would cause confusion.
The words ``positive'' and ``negative'' indicate the slopes 
of the lines on which the fibers lie in the Hasse diagram.)
More specifically, the element $(k,\ell )\in [a] \times [b]$ 
belongs to {\em rank} $k+\ell-2$, {\em file} $\ell-k$,
\emph{positive fiber} $k$, and \emph{negative fiber} $\ell$. 
\end{definition}

To each order ideal $I \in J([a] \times [b])$ we associate
a lattice path of length $a+b$ 
joining the points $(-a,a)$ and $(b,b)$ in the plane,
where each step is of type $(i,j) \rightarrow (i+1,j+1)$
or of type $(i,j) \rightarrow (i+1,j-1)$;
this path separates the squares in the modified Hasse diagram
corresponding to poset elements that lie in $I$
from the squares corresponding to poset elements that do not.
Here is a self-contained and concrete description.
Given $1 \leq k \leq a$ and $1 \leq \ell \leq b$,
represent $(k,\ell) \in [a]\times [b]$
by the square centered at $(\ell-k,\ell+k-1)$ with vertices 
$(\ell-k,\ell+k-2)$, $(\ell-k,\ell+k)$, 
$(\ell-k-1,\ell+k-1)$, and $(\ell-k+1,\ell+k-1)$.
(Thus, the poset-elements $(k,\ell) = (1,1)$, $(a,1)$, $(1,b)$, and $(a,b)$
are respectively the bottom, left, right, and top squares
representing elements of $[a] \times [b]$.)
Then the squares representing the elements of the order ideal $I$
form a ``Russian-style'' Young diagram
whose upper border is a path joining
some point on the line through the origin of slope $-1$
to some point on the line through the origin of slope $+1$.
Adding extra edges of slope $-1$ at the left
and extra edges of slope $+1$ at the right if necessary,
we get a path joining $(-a,a)$ to $(b,b)$.
See Figures~\ref{fig:pm32a},~\ref{fig:pm32b}, and~\ref{fig:rm42} 
for several examples of this correspondence.  

\begin{definition}\label{def:htfn}

We can think of this path as the graph of
a (real) piecewise-linear function $h_I: [-a,b]\rightarrow [0,a+b]$;
we call this function (or its restriction to $[-a,b] \cap \bZ$)
the {\bf height function} representation of the ideal $I$.
Equivalently, for every $k \in [-a, b]$, we have
\[
 h_I (k) = \left|k\right| + 2 \# \left(
 \text{elements of } I \text{ in file } k\right) .
\]
In particular, $h_I (-a) = a$ and $h_I (b) = b$.

To this height function we can in turn associate a word
consisting of $a$ $-1$'s and $b$ $+1$'s,
whose $i$th letter (for $1 \leq i \leq a+b$) is $h_I(i-a)-h_I(i-a-1)=\pm 1$;
we call this the {\bf sign-word} associated with the order ideal $I$.
\end{definition}

Note that the sign-word simply lists the slopes of the segments making
up the path, and that either the sign-word or the height-function encodes
all the information required to determine the order ideal.  

\begin{proposition}\label{prop:cardI}
Let $I\in J([a]\times [b])$ correspond with height function
$h_{I}:[-a,b]\rightarrow \bR$. Then 
\[
\sum_{k=-a}^{b}h_{I}(k) = \frac{a(a+1)}{2}+\frac{b(b+1)}{2}+2\#I\,.
\]

\end{proposition}
So to prove that the cardinality of $I$ is homomesic, it suffices to
prove that the function $h_I(-a)+h_I(-a+1)+\dots+h_I(b)$ is homomesic
(where our combinatorial dynamical system acts on height functions $h$
via its action on order ideals $I$).

\subsection{Promotion in products of two chains}
\label{subsec-promotion}

Given a ranked poset $P$, there always exists a (not necessarily
injective) map from $P$ to $\bZ \times \bZ$ that allows files to
be defined; given such a map (an rc-embedding, in the terminology of
Striker and Williams), it follows from Proposition~\ref{prop:togtwo}
that all toggles corresponding to elements within the same file commute
so their product is a well defined operation on $J(P)$.  
This allows one to define an operation on $J(P)$
by successively toggling all the files from left to right, 
in analogy to Corollary~\ref{cor:togranks}.   

From here on, we set $P = [a] \times [b]$, 

\begin{theorem}[{Striker-Williams~\cite[\S~6.1]{SW12}}]\label{thm:sw}
Let $x_1,x_2,\dots,x_n$ be any enumeration
of the elements $(k,\ell)$ of the poset $[a] \times [b]$
arranged in order of increasing $\ell - k$.
Then the action on $J(P)$ given by $\promotion:=
\sigma_{x_n} \circ \sigma_{x_{n-1}} \circ \cdots \circ \sigma_{x_1}$
viewed as acting on the paths (or the sign-words
representing them) is just a leftward cyclic shift.  
\end{theorem}

Striker and Williams call this well-defined composition $\promotion $
\textbf{promotion} (since it is related to Sch\"utzenberger's
notion of promotion on linear extensions of posets). They show
that it is conjugate to rowmotion in the toggle group, obtaining a
much simpler bijection to prove Panyushev's conjecture in Type A, and
generalizing an equivariant bijection for $[a]\times [b]$ of
Stanley~\cite[remark after Thm~2.5]{Sta09}.  
This definition and their results apply more generally to a class
they defined, initially called \textit{rc-posets} but later renamed
\textit{rc-embedded posets}, whose elements fit neatly into 
``rows'' and ``columns'' (which we call here ``ranks'' and ``files'').  
As with $\rowmotion$, we can think of $\promotion$ as operating either
on $J(P)$ or $\cA (P)$, adding subscripts $\promotion_{J}$ or
$\promotion_{A}$ if necessary.  
Since the cyclic left-shift has order $a+b$, so does $\promotion$.

\begin{figure}[t]
\begin{center}
\includegraphics[width=4.9in]{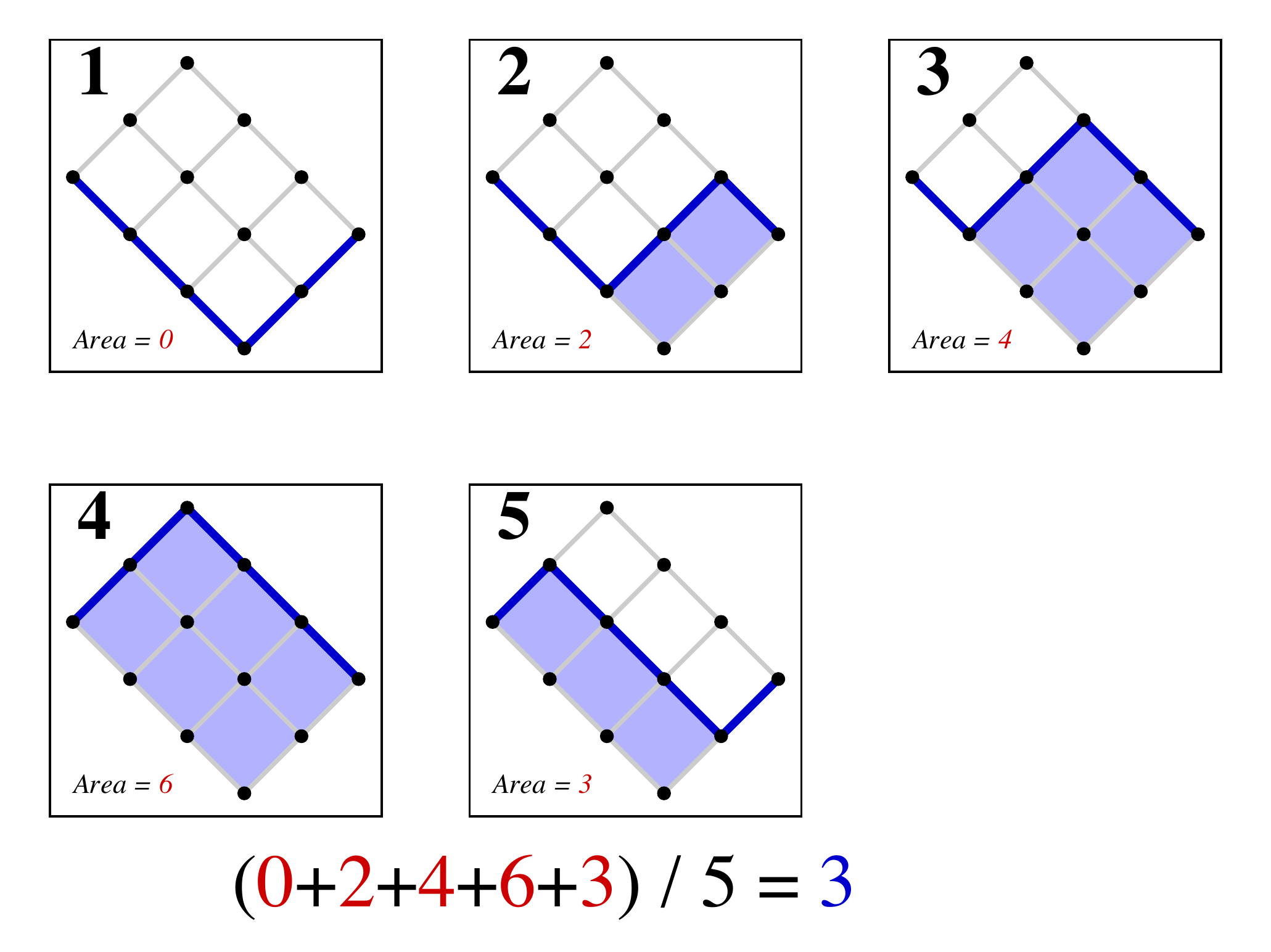}
\end{center}
\caption{One promotion orbit in $J([3]\times [2])$}
\label{fig:pm32a}
\end{figure}

\begin{figure}[t]
\begin{center}
\includegraphics[width=4.9in]{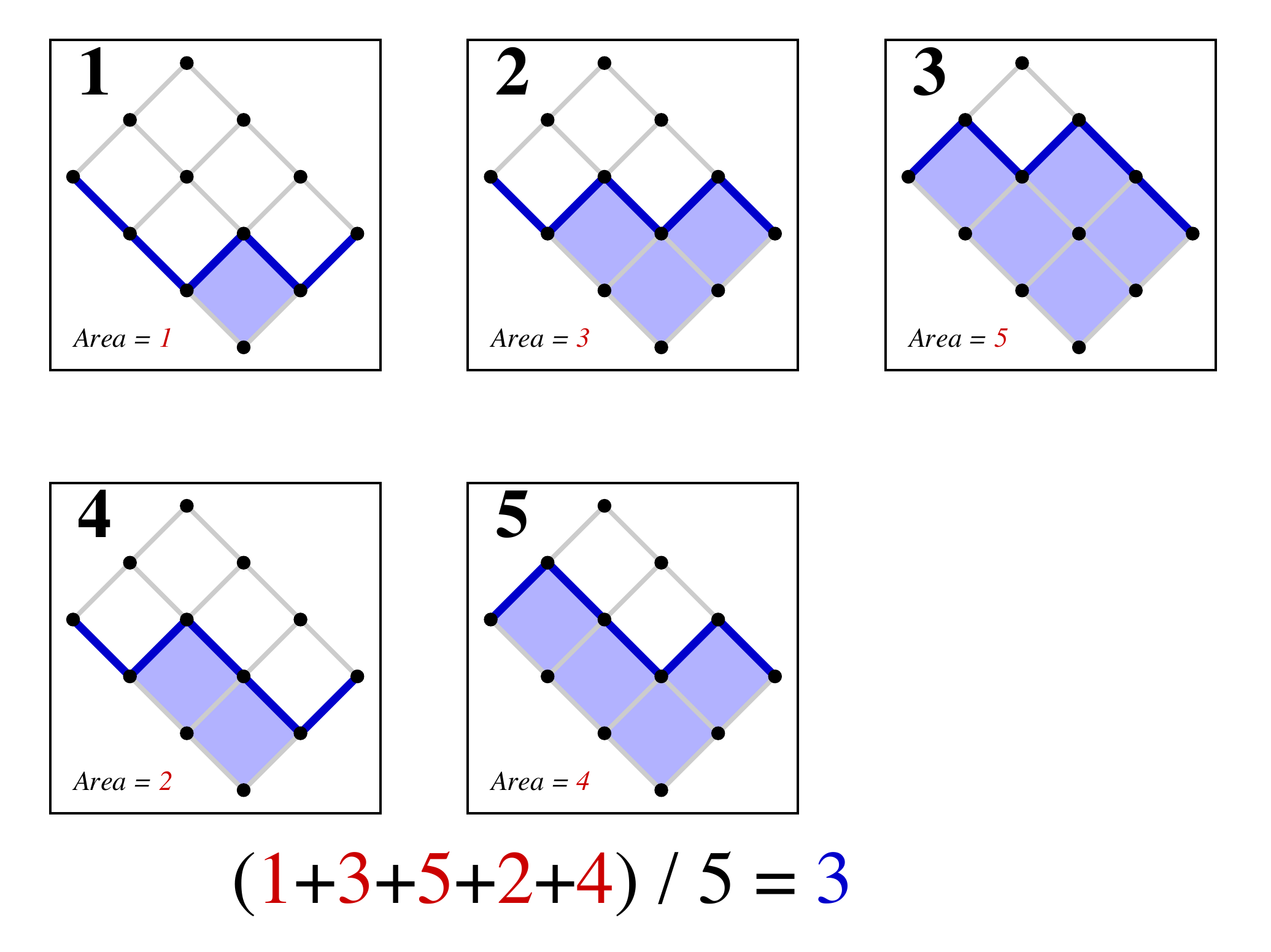}
\end{center}
\caption{The other promotion orbit in $J([3]\times [2])$}
\label{fig:pm32b}
\end{figure}

\begin{theorem}
\label{promotion-ideals}
The cardinality statistic is $c$-mesic under the action of 
promotion $\promotion_J$ on $J([a]\times [b])$, 
with $c={ab}/{2}$.  
\end{theorem}

\begin{proof}
To show that $\#I$ is homomesic, by Proposition~\ref{prop:cardI} it
suffices to show that $h_I(k)$ is homomesic 
for all $-a \leq k \leq b$.  
Note that here we are thinking of $I$ as varying over $J(P)$,
and $h_I(k)$ (for $I$ varying) as being a
real-valued function on $J(P)$.

We can write $h_I(k)$ as the telescoping sum
$h_I(-a) + (h_I(-a+1)-h_I(-a)) + (h_I(-a+2)-h_I(-a+1)) 
+ \dots + (h_I(k)-h_I(k-1))$;
to show that $h_I(k)$ is homomesic for all $k$,
it will be enough to show that
all the increments $h_I(k)-h_I(k-1)$ are homomesic.
Note that these increments are precisely the letters of the sign-word of $I$.
Create a square array with $a+b$ rows and $a+b$ columns,
where the rows are the sign-words of $I$
and its successive images under the action of $\promotion$;
each row is just the cyclic left-shift of the row before.
Here for instance is the array for the example in Figure~\ref{fig:pm32a}
with $P$ the poset $[3] \times [2]$ and $I$ the empty order ideal:
$$
\begin{array}{ccccc}
- & - & - & + & + \\
- & - & + & + & - \\
- & + & + & - & - \\
+ & + & - & - & - \\
+ & - & - & - & + 
\end{array}
$$
Since each row contains $a$ $-1$'s and $b$ $+1$'s,
the same is true of each column.
Thus, for all $k$, the average value of the $k$th letters of the sign-words
of $I$, $\promotion I$, $\promotion^2 I$, \dots, $\promotion^{a+b-1} I$
is $(b-a)/(b+a)$.
This shows that the increments are homomesic, as required,
which suffices to prove the theorem.
\end{proof}

Our proof actually shows the more refined result that the restricted cardinality
functions  $\# (I \cap S)$ where $S$ is any file of $[a] \times [b]$ are
homomesic with respect to the action of $\promotion_{J}$.  

\begin{remark}\label{rem:corprop-bits}
We now have a third proof of Proposition~\ref{prop:bits}.  The bijection
sending $I \in J([a] \times [b])$ to its sign-word is an
isomorphism between promotion acting on order ideals in $[a]\times [b]$ 
and the leftward cyclic shift acting on the sign-word. 
Furthermore, the cardinality of any order ideal is mapped to
the number of inversions of the sign-word. 
So the homomesy of Theorem \ref{promotion-ideals}
yields the homomesy of Proposition~\ref{prop:bits}.  
\end{remark}

The next example shows that the cardinality of the antichain $A_I$
associated with the order ideal $I$
is {\em not\/} homomesic under the action of promotion $\promotion$.  
\begin{example}\label{eg:procA}
Consider the two promotion orbits of $\promotion_{A}$ shown in
Figures~\ref{fig:pm32a} and \ref{fig:pm32b}.  Although the statistic $\#I$ is
homomesic, giving an average of 3 in both cases, 
the statistic $\#A$ averages 
to $\frac{1}{5}\left(0+1+1+1+1\right)=\frac{4}{5}$ in the first orbit 
and to $\frac{1}{5}\left(1+2+2+1+2\right)=\frac{8}{5}$ in the second.  
\end{example}

\subsection{Rowmotion in products of two chains}
\label{subsec-rowmotion}

\begin{figure}[t]
\begin{center}
\includegraphics[width=4.7in]{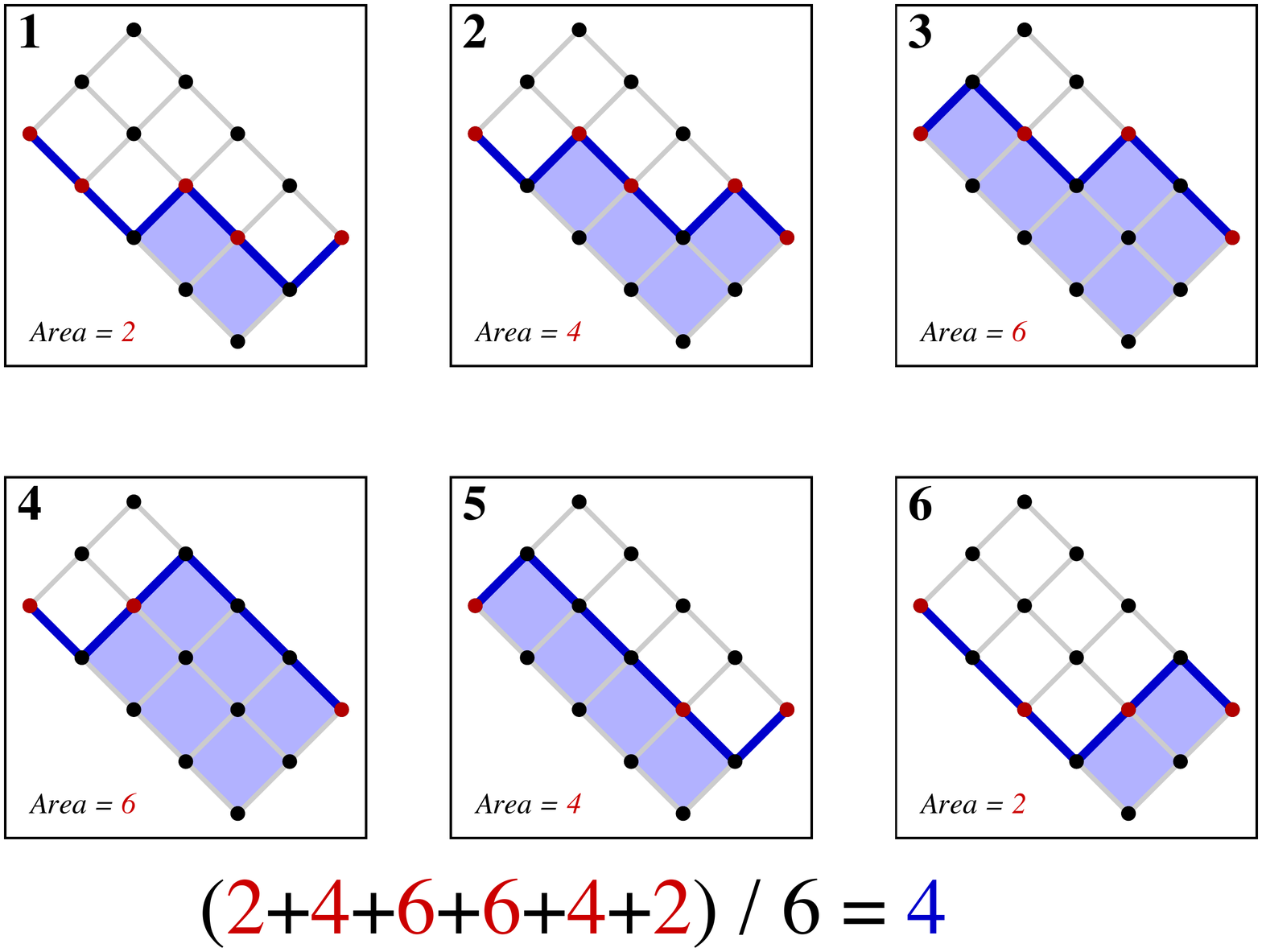}
\end{center}
\caption{A rowmotion orbit in $J([4]\times [2])$}
\label{fig:rm42}
\end{figure}

Unlike promotion, rowmotion turns out to exhibit homomesy 
with respect to both the statistic that counts the size of an order ideal 
and the statistic that counts the size of an antichain.

\subsubsection{Rowmotion on order ideals in $J([a] \times [b])$}\label{sss:rmJP}

We can describe rowmotion nicely in terms of the sign-word.
We define a {\bf block} within any word $w\in \{-1, +1\}^{n}$
to be an occurrence of the factor $-1,+1$
(that is, a $-1$ followed immediately by a $+1$).
A \textbf{gap} in the sign-word is a factor which contains
no block; in other words, it is a factor of the form
$+1,+1,\ldots,+1, -1,-1,\ldots,-1$.
This uniquely decomposes any sign word into blocks and gaps.  

Now define the \textbf{block-gap reversal} of $w$ to be 
the word $\tilde{w}$ obtained by
decomposing $w$ into contiguous block and gap subwords, 
then reversing each subword (leaving
the subwords in the same relative order).  
For example, the binary word 
\begin{align*}
w &= -1, +1,   +1, -1, -1,   -1, +1,   +1 \qquad \qquad 
\text{\!\!\!\!\!is divided into blocks and gaps as}&\hfil \\
  & = -1, +1,\cred{|} +1, -1, -1, \cred{|} -1, +1, \cred{|} +1 \,. \qquad 
\text{\!\!\!\!\!\!\!\!\!Reversing each block and gap in place gives}&\hfil \\
\tilde{w} &= +1, -1, \cred{|} -1, -1, +1, \cred{|} +1, -1, \cred{|} +1 \qquad 
\text{\!\!\!\!\!\!or dropping the dividers}&\hfil \\
          &= +1, -1, -1, -1, +1, +1, -1, +1\,. & & \\ 
\end{align*}

\begin{lemma}\label{lem:row180}
Let $I\in J([a]\times [b])$ correspond to the sign-word $w$, 
and let $\tilde{w}$ be the block-gap reversal of $w$.  
Then the sign-word of $\Phi_{J}(I)$ is $\tilde{w}$.  
In other words, rowmotion on order ideals is equivalent 
to block-gap reversal on corresponding sign-words.  
\end{lemma}

Note that (in general) the dividers correspond to the red dots 
in Figure~\ref{fig:rm42}, so one can visualize $\rowmotion_{J}$ as
reversing ($180^{\circ}$ rotation of)
each lattice-path segment that corresponds to a block or a gap in the
sign-word.  This is illustrated in Michael LaCroix's animations, (which require Adobe
acrobat),  within talk slides at
\url{http://www.math.uconn.edu/~troby/combErg2012kizugawa.pdf}.  
\begin{proof}
Consider Figure~\ref{fig:rm42}, where the elements of the poset 
are denoted by the \emph{squares} (not the dots), 
and the shaded portions indicate the order ideals to which
rowmotion is being applied.  
Note that the sign-word of $I$ indicates the lattice path that
traces out the boundary between $I$ and its complement $I^{C}$.  
For example, in the second picture, 
the lattice path follows the pattern $-1,+1,-1,-1,+1,-1$.  
Clearly the minimal elements of the complement $I^{C}$ 
occur exactly in the locations \emph{above} where 
we have a $-1,+1$ pair (indicating a down step followed
by an up step as we move from left to right along the lattice path).  
By definition of rowmotion, these squares 
become the generators of $\Phi_{J}(I)$.  
In particular, each block $-1,+1$ will map to its reversal $+1, -1$, 
so that these minimal elements of $I^{C}$ are now
\emph{maximal} in $\Phi_{J}(I)$.  

Now let $G$ be any gap occurring between two blocks $B$ and $B'$ 
corresponding to the minimal elements $(i,j)$ and $(i',j')$ in $I^{C}$.  
Then $G$ must consist of $j'-j-1\geq 0$ up steps, 
followed by $i-i'-1 \geq 0$ down steps 
(since the two minimal elements are incomparable).  
Now by definition of rowmotion, $(i,j)$ and $(i',j')$ are two
adjacent maximal elements of $\Phi_{J}(I)$, and so
the part of the sign-word of $\Phi_{J}(I)$ between
the corresponding $+1, -1$ segments will have the
form $-1,-1,\ldots,-1, +1,+1,\ldots,+1$.
Thus the lattice path segment that corresponds to this part
consists of $i-i'-1$ down steps followed by $j'-j-1$ up steps
(this is especially clear if one creates a generic diagram 
like those in Figure~\ref{fig:rm42}).
Similar arguments handle the cases where the gap occurs at the
beginning or end of the sign-word.

\end{proof}

An alternative way to compute the block-gap reversal of a word $w$ is to 
(1) prepend a $+$ and append a $-$, obtaining a new word $w'$;
(2) exchange the $i$th run of $+$'s in $w'$
with the $i$th run of $-$'s, for all applicable $i$, 
obtaining a new word $w''$; and
(3) delete the initial $-$ and terminal $+$ in $w''$.

It turns out that all we really need to know for purposes of proving homomesy
is that the sign-word for $I$ has $-1,+1$ in a pair of adjacent positions
if and only if the sign-word for $\rowmotion_{J}(I)$ has $+1,-1$ 
in the same two positions.
This can be seen directly for $J([a]\times [b])$ from the description of 
$\Phi_J$ given at the start of Section~\ref{sec:chains}.  (See also
Figure~\ref{fig:rm42}.)  This situation occurs if and only if
the antichain $A(\rowmotion_{J}(I))$ contains an element
in the associated file of $[a] \times [b]$.

\begin{theorem}
\label{rowmotion-ideals}
The cardinality statistic is $c$-mesic
under the action of rowmotion $\rowmotion_J$ on $J([a]\times [b])$, 
with $c={ab}/{2}$. 
\end{theorem}

\begin{proof}
As in the previous section, to prove that $\# I$ is homomesic under rowmotion,
it suffices to prove 
that all the increments $h_I(k)-h_I(k-1)$ are homomesic.  
There is a positive integer $N$ such that $\rowmotion^N = \id$
(since $J([a] \times [b])$ is finite\footnote{A result of
Fon-der-Flaass~\cite[Theorem 2]{Fon93} states that the size of any
$\rowmotion$-orbit in $[a]\times [b]$ is a divisor of $a+b$
(this also follows from Proposition~\ref{prop:awprop}), so 
that we can take $N = a+b$; but any $N$ will do here.}).
Now, proving that $h_I(k)-h_I(k-1)$ is homomesic
is equivalent to showing that for all $k$,
the sum of the $k$th letters of the sign-words of
$\rowmotion^0 I, \rowmotion^1 I, \ldots,
\rowmotion^{N-1} I$ is independent of $I$.  
Create a rectangular array with $N$ rows and $a+b$ columns,
where the rows are the sign-words of $I$
and its successive images under the action of $\rowmotion$.
Here for instance is the array for the example in Figure~\ref{fig:rm42}
with $P$ the poset $[4] \times [2]$ and $I$ the order ideal
generated by $(2,1)$:
$$
\begin{array}{cccccc}
- & - & + & - & - & + \\
- & + & - & - & + & - \\
+ & - & - & + & - & - \\
- & + & + & - & - & - \\
+ & - & - & - & - & + \\
- & - & - & + & + & -
\end{array}
$$
Consider any two consecutive columns of the array,
and the width-2 subarray they form.
There are just four possible combinations of values in a row of the subarray:
$(+1,+1)$, $(+1,-1)$, $(-1,+1)$, and $(-1,-1)$.
However, we have just remarked that a row is of type $(-1,+1)$ 
if and only if the next row is of type $(+1,-1)$
(where we consider the row after the bottom row to be the top row).
Hence the number of rows of type $(-1,+1)$
equals the number of rows of type $(+1,-1)$.
It follows that any two consecutive column-sums of the full array are
equal, since other row types contribute the same value to each column
sum. That is, within the original rectangular array, 
every two consecutive columns have the same column-sum.
Hence all columns have the same column-sum.
This common value of the column-sum must be $1/(a+b)$
times the grand total of the  values of the rectangular array.
But since each row contains $a$ $-1$'s and $b$ $+1$'s,
each row-sum is $b-a$, so the grand total is $N(b-a)$,
and each column-sum is $N(b-a)/(a+b)$.
Since this is independent of which rowmotion orbit we are in,
we have proved homomesy for letters of the sign-word of $I$
as $I$ varies over $J([a] \times [b])$,
and this gives us the desired result about $\# I$,
just as in the proof of Theorem \ref{promotion-ideals}.
\end{proof}

\subsubsection{Rowmotion on antichains in $\cA ([a]\times [b])$}\label{sss:rmAP}

In his survey article on promotion and evacuation, Stanley~\cite[remark
after Thm~2.5]{Sta09} gave a concrete equivariant bijection between
rowmotion $\rowmotion_{A}$ acting on antichains in $\cA ([a]\times [b])$ and
cyclic rotation of certain words on $\{1,2\}$. Armstrong (private
communication) gave a variant description that clarified the correspondence, 
which he learned from Thomas and which we use in what follows.  

\begin{figure}[t]
\begin{center}
\includegraphics[width=3.2in]{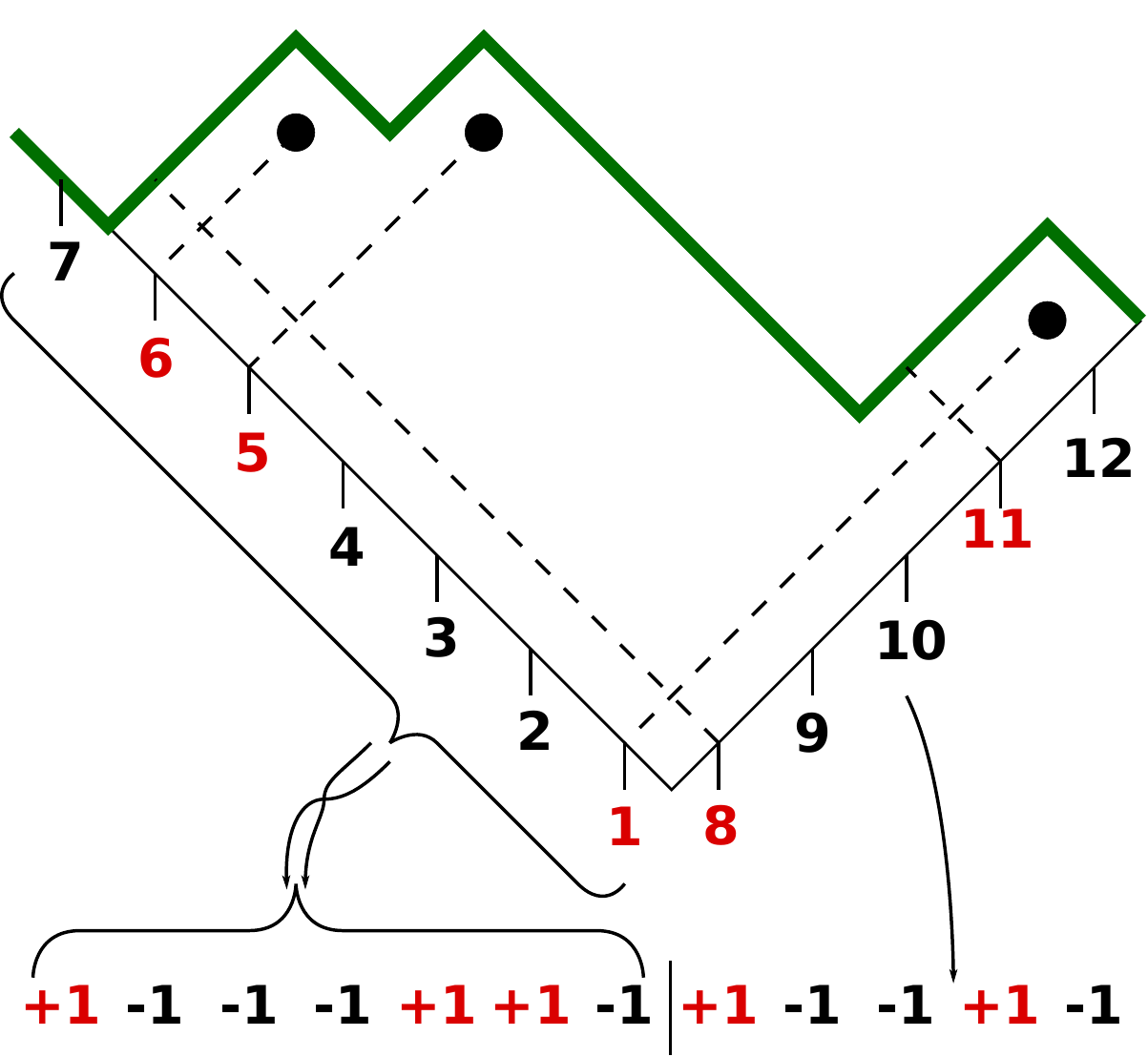}
\end{center}
\caption{The Stanley-Thomas word for a 3-element antichain 
in $\cA ([7]\times [5])$.  Red and black correspond to $+1$
and $-1$ respectively.}
\label{fig:aw75}
\end{figure}

\begin{definition}\label{def:aw}

Fix $a$, $b$, and $n=a+b$.
For every given $k \in [a]$, we call the subset
$\{(k,\ell): \ \ell \in [b]\}$ of $[a] \times [b]$
the \textbf{$\mathbf{k}$th positive fiber}.
For every given $\ell \in [b]$, we call the subset
$\{(k,\ell): \ k \in [a]\}$ of $[a] \times [b]$
the \textbf{$\mathbf{\ell}$th negative fiber}.
Define the \textbf{Stanley-Thomas word} 
$w(A)$ of an antichain $A$ in $[a] \times [b]$
to be $w_{1}w_{2}\dotsb w_{a+b}\in \{-1,+1 \}^{a+b}$ with
\[
w_{i}:=\begin{cases}
+1,& \text{if $A$ has an element in \emph{positive fiber} $i$ 
(for $i\in [a]$) or}\\
\  & \ \ \ \text{$A$ has NO element in \emph{negative fiber} $i-a$ 
(for $a+1\leq i\leq n$);}\\ 
-1 & \text{otherwise.}\\ 
\end{cases} 
\]
\end{definition}

As usual, if $u$ is a word we write $u_i$ to denote its $i$th letter.

\begin{example}\label{eg:aw}
As illustrated in Figure~\ref{fig:aw75}, 
let $P = [7] \times [5]$ and $A=\{ (1,5), (5,3), (6,2)\}$.  
By definition, the Stanley-Thomas word $w(A)$ should have $+1$ in
entries 1, 5, and 6 (positive fibers where $A$ appears) and in entries 8 and 11 
(negative fibers where $A$ does not appear, with indices shifted by $7=a$).  
Indeed one sees that $w(A)=+1,-1,-1,-1,+1,+1,-1, \mid +1,-1, -1, +1, -1$
(where the divider between $a$ and $a+1$ is just for ease of reading).  
Note that applying rowmotion gives $A'=\rowmotion (A)=\{(2,4), (6,3), (7,1)\}$ 
with Stanley-Thomas word 
$w(A')= -1, +1,-1,-1,-1,+1,+1, \mid -1, +1, -1, -1, +1=C_{R}\,w(A)$, the
rightward cyclic shift of $w(A)$.  
\end{example}

\begin{proposition}[Stanley-Thomas]
\label{prop:awprop}
The correspondence $A\longleftrightarrow w(A)$ is a bijection from
$\cA ([a]\times [b])$ to binary words $w\in \{-1,+1 \}^{a+b}$ with exactly
$a$ occurrences of $-1$ and $b$ of $+1$.  
Furthermore, this bijection is equivariant
with respect to the actions of rowmotion $\rowmotion_{A}$ 
and rightward cyclic shift $C_{R}$.  
\end{proposition}

Note that the classical result that
$\Phi_A^{a+b}$ is the identity map follows immediately.

\begin{proof}
Let $\cW_{a,b}$ denote the set of binary words in $\{-1,+1 \}^{a+b}$ 
with exactly $a$ occurrences of $-1$ and $b$ of $+1$. 
The map $A\mapsto w(A)$ is clearly well-defined into $\{-1,+1 \}^{a+b}$.  
By definition, the number of occurrences of $+1$ among the first $a$
letters of $w(A)$ is $\#A$; among the remaining $b$ letters, it is $b-\#A$, 
giving a total of $b$ occurrences of $+1$ in $w(A)$.  
Thus $w(A)\in \cW_{a,b}$.  

This map has an inverse as follows.  Given any word $u \in \cW_{a,b}$,
let $k$ denote the number of indices $1 \leq i \leq a$ with $u_i = +1$,
and let $1 \leq i_{1} < i_{2} < \dots < i_{k} \leq a$ denote those indices. 
There must thus be $a-k$ indices $1 \leq i \leq a$ with $u_i = -1$,
and therefore $k$ indices $a+1 \leq j \leq a+b$ with $u_j = -1$
(since the total must sum to $a$ by definition of $\cW_{a,b}$).  
Let $a+1\leq j'_{1} < j'_{2} < \dots < j'_{k} \leq a+b$ denote those indices
corresponding to $-1$, and set $j_{\ell}: = j'_{\ell} - a$.  Then the
corresponding antichain $A$ is given by 
\[
A = \{(i_{1},j_{k}), (i_{2}, j_{k-1}),\dots , (i_{k}, j_{1}) \}\,.  
\]
(See Example~24.) 
Note that this is the only pairing of the indices that gives an
antichain in $[a]\times [b]$.  It follows from the definitions
that for this $A$, $w(A)=u$, whence $w$ is a bijection between $\cA
([a]\times [b])$ and $\cW_{a,b}$. 

It remains to show that the following diagram commutes:

\[
\xymatrix{
{\cA ([a]\times [b]) } \ar[r]^-w \ar[d]_{\Phi_{A}} & \cW_{a,b} \ar[d]^{C_{R}} \\
{\cA ([a]\times [b]) } \ar[r]_-w                   & \cW_{a,b}
}\,. 
\]
To that end, let $A$ be any antichain in $\cA ([a]\times [b])$, and set
$A':=\Phi (A)$, $u:=w(A)$ and $u' := w(A')$.  We want to show that $u' =
C_{R}u$. 

Recall our initial definition at the start of this section of $\Phi_{A}$ 
as the composition 
\[
\begin{aligned}
\Phi_{A}: & \cA (P) & \rightarrow  & \ \ J(P) & \rightarrow  & \ \ F(P) 
& \rightarrow & \ \ \cA (P) \\
	  &  A & \mapsto  & \ \ I_{A}  & \mapsto  & \ \ \overline{I_{A}} 
& \mapsto & \ \ A' \;\,, 
\end{aligned}
\]
where $A'$ is the set of minimal elements of the complement 
of the order ideal $I_{A}$ generated by $A$.  
Suppose first that $i\in [a-1]$.  We aim to show that $u'_{i+1} = u_i$.
If $u_{i}=+1$, then there is an
antichain element $(i,j) \in A$ in positive fiber $i$, 
which is not the top positive fiber.  
Because $A$ is an antichain, any element of $A$ in positive fiber $i+1$ 
must lie in a negative fiber $j'<j$. (This includes the case
when there is no element of $A$ in positive fiber $i+1$.)  
This means that the complement $\overline{I_{A}}$ of the
corresponding order ideal will have a minimal element in positive fiber $i+1$.  
(A glance at Figure~\ref{fig:aw75} should make this clear.)  Thus, by
definition of $\Phi_{A}$, $A'$ will have an element in positive fiber $i+1$, 
so $u'_{i+1}=+1$. 

On the other hand, if $u_{i}=-1$, 
then no element of $A$ lies in positive fiber $i$.  
If $\overline{I_A}$ had a minimal element $(i+1,p)$ in positive fiber $i+1$, 
then $(i,p)$ would lie in $I_A$ and thus below an element of $A$.  
But said element would have to lie in positive fiber $i$ (since $(i+1,p)
\in \overline{I_A}$), contradicting the fact that no element of
$A$ lies in positive fiber $i$. Hence, $u'_{i+1} = -1$.  
So $u'_{i+1} = u_i$ in either case.

Similar arguments show that for $j\in [b-1]$, $u'_{a+j+1} = u_{a+j}$.
It remains only to check positions $a$ and $a+b$ in $u$.  

If $u_{a}=1$, then $I_{A}$ includes all of negative fiber 1.
Therefore, $\overline{I_{A}}$, and hence $A'$,  
has no elements at all in negative fiber 1, and
$u'_{a+1}=+1$ by definition.  On the other hand, if $u_{a}=-1$, then
$I_{A}$ includes only a proper subset of negative fiber 1.  
This means that $\overline{I_{A}}$, and hence $A'$, 
must have elements in negative fiber 1,
since the only elements smaller than an element in negative fiber 1 also lie in
negative fiber 1.  Thus, $u'_{a+1}=-1$. 

A similar argument shows that $u'_{1}=u_{a+b}$, and we have $u'=C_{R}\,u$
as required.  
\end{proof}

\begin{theorem}
\label{rowmotion-antichains}
The cardinality statistic is $c$-mesic
under the action of rowmotion $\rowmotion_A$ on $\cA ([a]\times [b])$, 
with $c=ab/(a+b)$.
\end{theorem}

\begin{proof}
It suffices to prove a more refined claim, namely,
that if $S$ is any fiber of $[a] \times [b]$,
the cardinality of $A \cap S$ is homomesic under the action
of rowmotion on $A$.  By the previous result, 
rowmotion corresponds to cyclic shift of the Stanley-Thomas word,
and the letters of the Stanley-Thomas word tell us which fibers 
contain an element of $A$ and which do not.  
Specifically, for $1 \leq k \leq a$, if $S$ is the $k$th positive fiber,
then $A$ intersects $S$ iff the $k$th letter 
of the Stanley-Thomas word is a $+1$.
Since the Stanley-Thomas word contains $a$ $-1$'s and $b$ $+1$'s,
the superorbit of $A$ of size $a+b$
has exactly $b$ elements that are antichains that intersect $S$.
That is, the sum of $\# (A \cap S)$ over the superorbit of size $a+b$
is exactly $b$, for each of the $a$ positive fibers of $[a] \times [b]$.
Summing over all the positive fibers, we see that
the sum of $\# A$ over the superorbit is $ab$.
Hence $\# A$ is homomesic with average $ab/(a+b)$.
\end{proof}

\section{Summary}

First we summarize what we know about the specific case 
of products of two chains, going beyond what is proved here
and including results that will be proved in follow-up articles
such as~\cite{EP13}.
Then we discuss how the case of $[a] \times [b]$ can be conceived of 
as a small component of a larger research program.
Lastly, we offer some thoughts about directions
that this research program might take.

\subsection{Rowmotion and promotion for order ideals and antichains}

A natural way to find homomesies for the action of a map $\tau$
on some combinatorial set $\cS$
is to start with some finite set 
of not necessarily homomesic functions $f_1,f_2,\dots,f_N$
associated with the combinatorial presentation of the set $\cS$,
and then to inquire which linear combinations of the $f_i$'s are homomesic.
For example, if $\cS$ is the set of order ideals of a poset $P$,
then for each element $x \in P$ we have
an indicator function $1_x: \cS \rightarrow \{0,1\}$
such that $1_x (I)$ is 1 if $x \in I$ and 0 otherwise.
We look in the span of the functions $f_i$ (call it $V$);
the functions in $V$ that satisfy homomesy form a subspace of $V$
whose intersection with the subspace of invariant functions in $V$
consists only of the constant functions.

In the case of rowmotion acting on order ideals of $[a] \times [b]$,
we find that the function $\sum_{x \in F} 1_x$ is homomesic
whenever $F$ is a file of $[a] \times [b]$.
Also, $1_x + 1_y$ is homomesic whenever $x$ and $y$ are
opposite elements of $[a] \times [b]$
(that is, they are obtained from one another by
rotating the poset 180 degrees about its center).
These can be shown to generate the subspace of homomesies.

The situation is the same for promotion 
acting on order ideals of $[a] \times [b]$.
That is because of the extremely intimate relationship between
rowmotion and promotion, as seen for instance in
Theorem 5.4 of~\cite{SW12}.

In the case of rowmotion acting on antichains of $[a] \times [b]$,
the situation is different.
Now $\cS$ is the set of antichains of a poset $P$,
and for each element $x \in P$ we have
an indicator function $1_x: \cS \rightarrow \{0,1\}$
such that $1_x (A)$ is 1 if $x \in A$ and 0 otherwise.
Although this vector space, like the one considered above, is $|P|$-dimensional,
there is no way to write the $|P|$ indicator functions we have just defined
as linear combinations of the $|P|$ indicator functions considered above.
Hence there is no reason to expect the subspace of homomesies for antichains
to have anything to do with the subspace of homomesies for order ideals.

For rowmotion on antichains,
we find that the function $\sum_{x \in F} 1_x$ is homomesic
whenever $F$ is a fiber of $[a] \times [b]$.
Also, $1_x - 1_y$ is 0-mesic whenever $x$ and $y$ are
opposite elements of $[a] \times [b]$.
These can be shown to generate the subspace of homomesies.

For promotion on antichains, the situation is not so clear.
One thing we do know is 
that the homomesic subspace under the action of promotion
is {\it not} the same as the homomesic subspace under the action of rowmotion
(Theorem 5.4 of~\cite{SW12} cannot be applied here).
In particular, the total cardinality statistic is not homomesic in this case.
However, other statistics are homomesic.
A natural open problem is to settle this fourth case.

More broadly one can ask the same sorts of questions
when $[a] \times [b]$ is replaced by other rc-embedded posets 
in the terminology of~\cite{SW12}.
Preliminary work by the authors and others suggests that
typically the subspace of homomesies is substantial.

It should be stressed that the choice of an ambient space of statistics
plays a key role in determining what one finds.
The action of rowmotion on order ideals is conjugate to
the action of rowmotion on antichains,
but in choosing between the order ideals picture and the antichains picture
one is choosing between two different spaces of statistics
(one generated by the indicator functions arising from order ideals
and the other generated by the indicator functions arising from antichains).
Since these are two different spaces,
their homomesic subspaces can be (and are) different.
As a more trivial example,
note that if one considers the (huge) vector space
spanned by the indicator functions $1_s: \cS \rightarrow \{0,1\}$ ($s \in \cS$)
such that $1_s(s')$ is 1 if $s=s'$ and 0 otherwise,
then the space of homomesies is large but not very interesting,
as it reflects only the orbit-structure of the action,
in a very simple way.

\subsection{Cyclic sieving}

We have observed informally that the sorts of combinatorial objects
that exhibit the cyclic sieving phenomenon~\cite{RSW04, RSW14}
also tend to exhibit the ``homomesy phenomenon''
(by which we mean, the abundance of homomesies).
It is natural to ask whether the connection goes both ways.
We think the answer is No.
Specifically, we can construct examples of (conjectural) homomesy
in which the order of the cyclic group generated by $\tau$
is much larger than the size of $\cS$
(e.g., $|\cS|=377$ while the order of $\tau$ exceeds 3 million).
This is very unlike typical instances of the CSP,
for which we have actions of small cyclic groups on large combinatorial sets.

\subsection{Equivariant bijections}
\label{subsec-equivariant}

Given the role that equivariant bijections play 
in the proofs of homomesy results,
one might come to the view that the bijections are what is truly fundamental,
while the homomesies are epiphenomena.
We have some sympathy for this point of view.
Those leaning in this direction should view homomesies as empirical indicators 
of the existence of (known or unknown) equivariant bijections
whose unearthing renders the homomesies explicable.  

However, it should be borne in mind that some homomesy results
do not follow from the existence of a single equivariant bijection,
but from the existence of many of them;
that is, sometimes a function is shown to be homomesic
by breaking it down as a linear combination of components
that are separately homomesic,
where different components require different equivariant bijections.
It's also worth noting that not all equivariant bijections 
used to prove homomesy are with objects that are being cyclically
rotated, e.g., Lemma~\ref{lem:row180}.  
Above all, the homomesy point of view brings to the fore
the notion that homomesies form a vector space.

\subsection{Complementarity}

In section 2 we saw several cases in which the ambient vector space $V$
(consisting of real-valued functions on $\cS$) can be written as the 
direct sum of the subspace of 0-mesic functions and the subspace of 
invariant functions under the action of $\tau: \cS \rightarrow \cS$.
These were the cases in which $V$ was closed under $\tau$
in the sense that, for every $f$ in $V$, 
$f \circ \tau$ is also in $V$
(so that in fact $f \circ \tau^k$ is in $V$ for all $k \geq 0$).
In that situation, the complementarity between 0-mesy and invariance
can be seen as a special case of the complementarity
between the image and the kernel of a projection map.

This kind of sharp complementarity between 
the notions of 0-mesy and invariance was not seen in section 3.
However, we could recover complementarity by suitably enlarging $V$
so that it includes the indicator functions of all events of the form
``the poset element $(k, \ell)$ belongs to $\rowmotion(I)$''
(or, in the case of actions on antichains, ``\ldots belongs to $\rowmotion(A)$).
It would be interesting to classify 0-mesies and invariants
of rowmotion and promotion in this setting.

\subsection{Promising avenues}

We have already mentioned that situations 
in which cyclic sieving has been observed 
have been (and mostly will continue to be)
good places to dig in search of homomesies.
One example is the CSP proved by Brendon Rhoades~\cite{R10}.

As another example, we mention the study of 
rowmotion on the product of three chains.
What are the homomesies for the action of rowmotion 
on order ideals or antichains in $[a] \times [b] \times [c]$?  
Preliminary study indicates that non-trivial homomesies exist
for generic $a$, $b$, and $c$.

Toggles as discussed in subsection~\ref{subsec-toggle} can be viewed
in the more general context of flipping in polytopes.
This point of view was first proposed (in a special case) in~\cite{KB95}
and is developed more fully in~\cite{EP13}.
Products of toggles in this geometrical setting
seem like a likely source of interesting homomesies.

It would be extremely interesting if homomesies showed up
in the discrete dynamical systems associated with cluster algebras.
The example of subsection~\ref{ssec:five}
suggests that cluster algebras of type A
(associated with frieze patterns)
might be a natural place to look.

\bibliographystyle{abbrvnat}
\bibliography{sample}
\label{sec:biblio}

\end{document}